\documentclass[aop]{imsart}
\pdfoutput=1
\usepackage{amsmath, amsthm, amssymb,bbm, tikz} %pakiety matematyczne
\usepackage{xcolor}
\usepackage[normalem]{ulem}
\usepackage[utf8]{inputenc} %kodowanie dokumentu
\usepackage{graphicx} % inserting pictures
\usepackage[small]{caption}
\usetikzlibrary{arrows,calc}
\numberwithin{equation}{section}

\usepackage[T1]{fontenc} %font encoding

\usepackage{enumerate} %enumeration
\usepackage{verbatim} %block comments
\usepackage{color}
\usepackage{tabularx}

\usepackage[colorlinks]{hyperref} %adds hyperlinks
\startlocaldefs
\theoremstyle{plain}

\newtheorem{theorem}{Theorem}[section]
\newtheorem*{theorem*}{Theorem}
\newtheorem{lemma}[theorem]{Lemma}
\newtheorem{claim}[theorem]{Claim}
\newtheorem{proposition}[theorem]{Proposition}

\newtheorem{corollary}[theorem]{Corollary}

\theoremstyle{remark}

\newtheorem*{definition*}{Definition}

\newtheorem*{question*}{Question}
\newtheorem*{example*}{Example}
\newtheorem*{examples*}{Examples}
\newtheorem{remark}[theorem]{Remark}
\newtheorem*{remark*}{Remark}

\newcommand{\norm}[1]{\left\Vert #1 \right\Vert}
\newcommand{\crochet}[1]{\left\langle #1 \right\rangle}
\renewcommand{\bar}[1]{\overline{#1}}

\newcommand{\Prob}{\mathbb{P}}

\newcommand{\R}{\mathbb{R}}
\newcommand{\E}{\mathbb{E}}

\newcommand{\N}{\mathbb{N}}
\newcommand{\calD}{\mathcal{D}}
\newcommand{\calN}{\mathcal{N}}
\renewcommand{\d}{\mathrm{d}}
\renewcommand{\P}{\Prob}
\renewcommand{\S}{\mathbb{S}}
\newcommand{\sd}{\mathbb{S}^{d-1}}

\newcommand{\ind}[1]{\mathbbm{1}_{\{ #1\}}}
\newcommand{\indset}[1]{\mathbbm{1}_{#1}}

\newcommand{\one}{\mathbbm{1}}
\newcommand{\given}{\;\big|\;}
\newcommand\abs[1]{\left|#1\right|}

\newcommand{\ep}{\epsilon}

\newcommand{\usim}{\mathop{\sim}_{\scriptscriptstyle(\mathrm{u})}}

\newcommand{\tl}{\tilde{t}}
\newcommand{\tll}{\tilde{t}-\ell}

\newcommand\UB[2]{\overline{\cB}_{#1}^{#2}}
\newcommand\LB[2]{\underline{\cB}_{#1}^{#2}}
\newcommand{\x}{\mathbf{x}}
\newcommand{\y}{\mathbf{y}}
\newcommand{\fC}{\mathfrak C}
\newcommand{\fT}{\mathfrak T}
\newcommand{\win}{{\normalfont\texttt{win}}}

\newcommand{\B}{\mathcal{B}^{\Bumpeq}}

\newcommand{\fM}{\mathfrak{M}}

\newcommand{\cB}{\mathcal B}

\newcommand{\fE}{\mathfrak{E}}
\newcommand{\cE}{\mathcal{E}}
\newcommand{\cG}{\mathcal{G}}
\newcommand{\cF}{\mathcal F}
\newcommand{\cN}{\mathcal{N}}
\newcommand{\dd}{\mathrm{d}}
\newcommand{\bfx}{\mathbf{x}}

\begin{document}

\begin{frontmatter}
\title{The extremal point process of branching Brownian~motion~in~\texorpdfstring{$\R^{\text{\lowercase{d}}}$}{Rd}}
\runtitle{The extremal point process of BBM in~{$\R^{\text{\lowercase{d}}}$}}

\begin{aug}
  \author[A]{\fnms{Julien} \snm{Berestycki}\ead[label=e1]{julien.berestycki@stats.ox.ac.uk}},
  \author[B]{\fnms{Yujin H.} \snm{Kim}\ead[label=e2]{yujin.kim@courant.nyu.edu}},
  \author[C]{\fnms{Eyal}  \snm{Lubetzky}\ead[label=e3]{eyal@courant.nyu.edu}},\\
  \author[D]{\fnms{Bastien} \snm{Mallein}\ead[label=e4]{mallein@math.univ-paris13.fr}}
  \and
  \author[E]{\fnms{Ofer} \snm{Zeitouni}\ead[label=e5]{ofer.zeitouni@weizmann.ac.il}}

  \runauthor{J. Berestycki, Y.H. Kim, E. Lubetzky, B. Mallein and O. Zeitouni}

 %Julien
 \address[A]{Department of Statistics and Magdalen College,
University of Oxford, UK, \printead{e1}}

 %Yujin
 \address[B]{Courant Institute of Mathematical Sciences, New York University, \printead{e2}}

 %Eyal
 \address[C]{Courant Institute of Mathematical Sciences, New York University, \printead{e3}}

 %Bastien
 \address[D]{LAGA UMR 7539, Universit\'e Sorbonne Paris Nord, \printead{e4}}

 %Ofer
 \address[E]{Department of Mathematics, Weizmann Institute of Science, \printead{e5}}
 \end{aug}

 \begin{abstract}
 We consider a branching Brownian motion in $\R^d$ with $d \geq 1$ in which the position $X_t^{(u)}\in \R^d$ of a particle $u$ at time $t$ can be encoded by its direction $\theta^{(u)}_t \in \S^{d-1}$ and its distance $R^{(u)}_t$ to 0. We prove that the \emph{extremal point process}  $\sum \delta_{(\theta^{(u)}_t, R^{(u)}_t - m_t^{(d)})}$ (where the sum is over all particles alive at time $t$ and $m^{(d)}_t$ is an explicit centering term) converges in distribution to a randomly shifted, decorated Poisson point process on $\S^{d-1} \times \R$. More precisely, the so-called {\it clan-leaders} form a Cox process with intensity proportional to $D_\infty(\theta) e^{-\sqrt{2}r} \d r\d \theta $, where~$D_\infty(\theta)$ is the limit of the derivative martingale in direction $\theta$ and the decorations are i.i.d.\ copies of the decoration process of the standard one-dimensional  branching Brownian motion.
 This proves a conjecture of Stasi\'nski,  Berestycki and Mallein (Ann.\ Inst.\ H.\ Poincar\'{e} 57:1786--1810, 2021). The proof builds on that paper and
 on Kim, Lubetzky and Zeitouni (Ann. Appl. Prob. 33(2):1315--1368, 2023).
\end{abstract}

\begin{keyword}[class=MSC]
\kwd[Primary ]{60J80}
\kwd{60G70}
\kwd[; secondary ]{60J60, 60G15}
\end{keyword}

\begin{keyword}
\kwd{Branching Brownian motion}
\kwd{Cluster process}
\kwd{Decorated Poisson point
process}
\kwd{Extremal point process}
\kwd{Extreme value theory}
\end{keyword}

\end{frontmatter}

\section{Introduction}

A (binary) branching Brownian motion (BBM) in dimension $d \geq 1$ is a continuous-time branching particle system in which every particle moves independently as a Brownian motion in dimension $d$ and branches at rate $1$ into two daughter particles. For all $t \geq 0$ we write $\mathcal{N}_t$ for the set of particles alive at time $t$, and for $u \in \mathcal{N}_t$ we set $X_s^{(u)} \in \R^d$ to be the position at time $s \leq t$ of particle $u$ or its ancestor alive at that time. In this article, we will describe the structure of the limiting extremal point process, i.e. the particles that have travelled the furthest from the origin. For all $x \in \R^d$, we denote by $\P_x$ the law of the branching Brownian motion such that the initial particle starts from position $x$, with the convention $\P=\P_0$.

The study of extremal particles in dimension $d=1$ traces its roots to the work of Fisher~\cite{Fisher37}, Kolmogorov, Petrovskii and Piskunov~\cite{KPP37}, and McKean~\cite{McKean75}, and is by now well understood: indeed, using~\cite{KPP37}, McKean~\cite{McKean75} showed that  the {\it rightmost position} $M(t) = \max_{u \in \mathcal{N}_t} X_u(t)$, centered at the median of its law, converges in distribution. The seminal work of Bramson~\cite{Bramson78,Bramson83} identified the centering $m^{(1)}_t = \sqrt 2 t - \frac 3{2\sqrt 2} \log t$, and introduced the method of truncated second moment through barriers. Lalley and Sellke~\cite{LS87}  were then able to prove that
\[
 \lim_{t \to \infty} \P(M(t) -m^{(1)}_t \le y) = \E \exp \{-CD_\infty e^{-\sqrt 2 y} \} \,,
\]
where $C$ is a positive constant and $D_\infty$ is an a.s. positive random variable, constructed as the almost sure limit of the so-called derivative martingale associated to the branching Brownian motion.
Hence, the limiting law of the centered maximum $M(t) -m^{(1)}_t$ is the law of a Gumbel random variable with the random shift  $ \log (CD_\infty)/\sqrt 2$, and in fact, it follows from
\cite{LS87} that $M(t)$ has Gumbel fluctuations around $m^{(1)}_t-  \log (CD_\infty)/\sqrt 2$ (see also~\cite{Aidekon13} for extensions to branching random walks).

Finally,  it was shown (independently and around the same time) by~\cite{ABBS12} and~\cite{ABK13} that the {\it extremal point process} converges in distribution
\[
  \lim _{t \to \infty} \sum_{u \in \calN_t} \delta_{X^{(u)}_t  - m^{(1)}_t } \ =: \  \mathcal L\,,
\]
where the limit point process $\mathcal L$  can be described as follows. It is shown in~\cite{ABK13,CHL19} that conditionally on $M(t)\ge \sqrt 2 t$ (which is an unusually large displacement), the extremal point process seen from $M(t)$ converges to a limit object $\calD$. More precisely:
\begin{equation}
  \label{eqn:defDecoration}
  \lim_{t \to  \infty}  \P \left( \sum_{u \in \calN_t} \delta_{X^{(u)}_t  - M(t) } \in \cdot \mid M(t)\ge \sqrt 2 t \right) = \P(\calD \in \cdot)
\end{equation}
and $\calD$ is called the {\it decoration point process}.
Let $(\chi_i)_{i \in \N}$ be the atoms of a Poisson point process on $\R$ with (random) intensity $C D_\infty e^{-\sqrt 2 y} \d y$ and $\{\calD^{(i)}\}_{i \in \N}$ be i.i.d.\ copies of $\calD$, then we have that in distribution
\[
  \mathcal L = \sum_{i \in \N} \sum_{z\in \calD^{(i)}} \delta_{ \chi_i+ z }\,,
\]
which is called a \textit{randomly shifted decorated Poisson point process}, using the terminology from~\cite{SuZ15}, and abbreviated to SDPPP($CD_\infty$, $e^{-\sqrt{2}y}\dd y$, $\mathcal{D}$).

In words, $\mathcal L$ is obtained by shifting each atom $\chi_i$ of a Poisson point process with intensity $e^{-\sqrt{2}y}\dd y$ by $\frac{1}{\sqrt{2}} \log CD_\infty$, and decorating it by an independent copy of $\calD$.
Besides their intrinsic interest, the results for one-dimensional branching Brownian motion (and their branching random walks counterparts)  have in recent years provided a road map for the analysis of other log-correlated fields, see e.g.
\cite{Biskup20} for a discussion of the two dimensional discrete GFF and~\cite{BH20} for the (non-Gaussian) sine-Gordon field.

\begin{figure}
     \begin{tikzpicture}[>=latex,font=\small]
  \node (fig) at (0,0) {
  \includegraphics[width=9 cm]{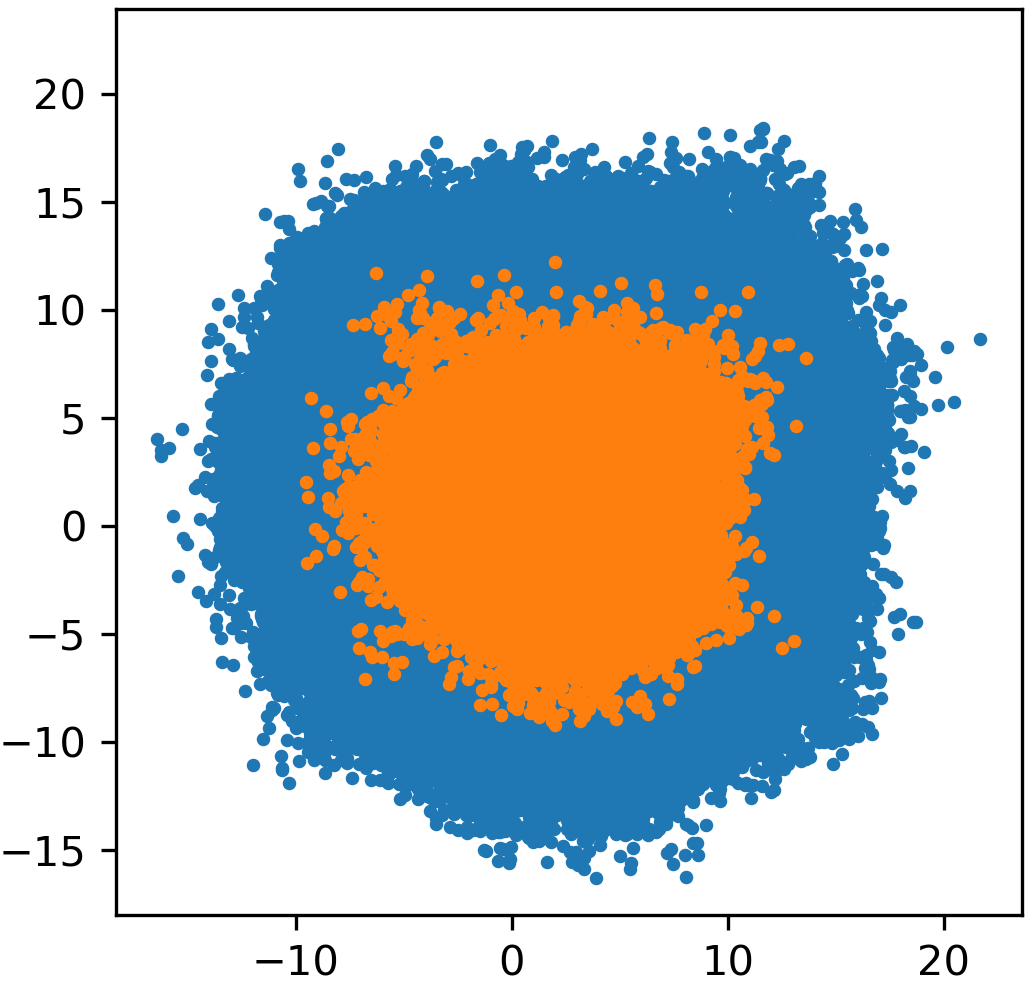}};
  \foreach \shx/\shy/\sca in {0/0/0.99, 0.25/-0.15/0.58}
  {\begin{scope}[scale=\sca, shift={(\shx,\shy)}]
  \draw[green, ultra thick, opacity=0.75] (-2.9,0) to[bend right=60] (1.5,-3.2) to[bend left=20] (3.35,-2.3)
  to [bend left=5] (3.6,1.3) to [bend left=5] (3, 2.5) to[bend right=10] (2.25, 3) to [bend right=5] (-1.4,2.8) to [bend right=5] (-1.8,2.6) to[bend right=20](-2.9,0);
  \end{scope}}
\end{tikzpicture}
\caption{A simulation of BBM, $d=2$, at times $t=10$ (orange) and $t=15$ (blue). The green curves depict (approximately) the centering term of the extremal point process $m^{(d)}_t + \frac{1}{\sqrt{2}}\log D_t(\theta)$ in polar coordinates.}
\end{figure}

By contrast, the case of the branching Brownian motion in dimension $d>1$ had until recently received far less attention. To describe what is known, we introduce
the polar decomposition for the position of particles. For all $x \in \R^d$, setting $\|x\|$ the Euclidian %$L^2$
 norm of $x$, we write
 % \propBM{\textbf{BM:} maybe the Euclidian norm to emphasize why we consider $L^2$ in particular?}
\[
  R_t^{(u)} = \|X_t^{(u)}\| \in [0,\infty), \quad  \quad \theta_t^{(u)} = \frac{X_t^{(u)}}{R_t^{(u)}} \in \S^{d-1}\,,
\]
and $R^*_t := \max_{u \in \cN_t} R_t^{(u)}$ for the largest Euclidean norm of a particle at time $t$. In~\cite{Big95}, Biggins proved that, whatever the dimension $d$,
\[
  \lim_{t \to \infty} \frac{R^*_t}{t} = \sqrt{2} \quad \text{a.s.,}
\]
and  Mallein~\cite{Mallein15} proved that, setting $m_t^{(d)} := \sqrt{2} t + \frac{d-4}{2\sqrt{2}} \log t$, the process $(R^*_t - m_t^{(d)}, t \geq 0)$ is tight. Then Kim, Lubetzky and Zeitouni~\cite{KLZ21} proved that $R^*_t - m_t^{(d)}$ converges in law to a Gumbel random variable, shifted by a constant plus an independent random variable $\log Z_\infty$, thus extending the aforementioned results of Bramson~\cite{Bramson83} and Lalley-Sellke~\cite{LS87} in dimension $1$ (however, in contrast with the situation in \cite{LS87}, the random variable $Z_\infty$ is not constructed as a measurable function of the branching Brownian motion -- this matter is resolved here in Corollary~\ref{cor:equiv-of-deriv-mtgs}).

The goal of the present paper is to obtain the full description of the limit extremal point process in dimension $d>1$, that is, to describe the limit of the random point measure on  $\S^{d-1} \times \R$ defined by
\[
  \mathcal{E}_t := \sum_{u \in \mathcal{N}_t} \delta_{(\theta_t^{(u)},R_t^{(u)} - m_t^{(d)})}\,.
\]
To do that, we first discuss  what plays the role of the random shift $D_\infty$. In~\cite{StBeMa2020}, Stasi\'nski, Berestycki, and Mallein introduced a multidimensional analogue of the derivative martingale: for all $t \geq 0$ and $\theta \in \S^{d-1}$,  set
\[
  D_t(\theta) =  \sum_{u \in \mathcal{N}_t} (\sqrt{2} t - X_t^{(u)}\cdot \theta) e^{\sqrt{2} X_t^{(u)} \cdot \theta - 2t} \text{ and }  D_\infty(\theta) = \max(0,\liminf_{t \to \infty} D_t(\theta)),
\]
where $u\cdot v$ is the usual inner product in $\R^d$ and $\{X_t^{(u)}, u \in \calN_t\}$ is a $d$-dimensional BBM (started from an arbitrary fixed point $x \in \R^d$; we do not carry the $x$-dependence in the notation, as it will be clear from the context). Observe that for each fixed $\theta$, the process $\{ X_t^{(u)}\cdot\theta, u \in \calN_t\} $ (the projection of the BBM on direction $\theta$) is just a standard one-dimensional BBM, and thus $D_t(\theta)$ is the usual associated derivative martingale.
They proved that $\P_x$-almost surely, there exists a random set $\Theta \subset \S^{d-1}$ of full Lebesgue measure such that~$D_t(\theta)$ converges to $D_\infty(\theta)$ for all $\theta \in \Theta$. Further, letting $\sigma$ denote the Lebesgue measure on $\S^{d-1}$, they show that
that the measure with density $D_t(\theta) \sigma(\d\theta)$  converges weakly to the measure with density $D_\infty(\theta) \sigma(\d\theta)$, $\P_x$-almost surely. Explicitly, for any bounded measurable functions $f,g : \S^{d-1} \to \R$, define
\[
\crochet{f,g} := \int_{\S^{d-1}} f(\theta) g(\theta) \sigma(\d\theta)\,;
\]
then it is shown in~\cite{StBeMa2020}  that for any bounded measurable $f: \S^{d-1}\to \R$,
\begin{align}
\lim_{t\to\infty} \crochet{D_t, f} = \crochet{D_{\infty}, f},\; \P_x\mathrm{-a.s.} \label{eqn:sbm-convergence}
\end{align}
As such, we will often view $D_t$ and $D_{\infty}$ as measures on $\S^{d-1}$ and write $D_t(A)$ and $D_{\infty}(A)$ (for $A \subset \S^{d-1}$) to denote $\crochet{D_t, \indset{A}}$ and $\crochet{D_{\infty}, \indset{A}}$, respectively. Observe that for any $a \in \R^d$, letting
\[
  D_t^{(a)}(\theta) =  \sum_{u \in \mathcal{N}_t} (\sqrt{2} t - ((X_t^{(u)}+a)\cdot \theta) e^{\sqrt{2} (X_t^{(u)}+a) \cdot \theta - 2t} \text{ and }  D_\infty^{(a)}(\theta) = \max(0,\liminf_{t \to \infty} D_t^{(a)}(\theta))
\]
we have $D_{\infty}^{(a)}(\theta) = e^{\sqrt{2}a\cdot \theta} D_{\infty}(\theta)$. Further, $D^{(a)}_{\infty}$ under $\P_x$ has the same law as $D_{\infty}$ under $\P_{x+a}$.

Our main theorem builds on~\cite{KLZ21} and~\cite{StBeMa2020} and describes the limit extremal point process; it answers in the affirmative Conjecture 1.4 from~\cite{StBeMa2020}.

\begin{theorem}
\label{thm:main}
For all $x \in \R^d$, under the law $\P_x$, the  extremal point process $\mathcal{E}_t$ converges in distribution
for the topology of vague convergence to a decorated Poisson point process $\mathcal{E}_\infty$ on $\S^{d-1} \times \R$ with the following description:
let $\alpha_d := (d-1)/2$ and $\gamma >0$ be the positive constant defined in~\eqref{def:gamma} below. Let $\{(\theta_i, \xi_i)\}_{i \in \N}$ be the atoms of a Poisson point process on $\S^{d-1} \times \R$ with intensity
\[
D_\infty(\theta) \sigma(\d \theta) \times
\gamma \pi^{-\alpha_d/2} \sqrt{2} e^{-\sqrt{2}y} \d y \,.
\]
Let  $\{\calD^{(i)}\}_{i \in \N}$ be i.i.d.\ copies of the decoration point process $\calD$ for the one-dimensional BBM
as in \eqref{eqn:defDecoration}.
Then
\[
  \mathcal{E}_\infty = \sum_{i = 1}^\infty \sum_{r \in \calD^{(i)}} \delta_{(\theta_i,\, \xi_i + r  )}\,.
\]
Furthermore, the convergence in distribution holds jointly with that of
$D_t(\cdot)$ to $D_\infty(\cdot)$.
\end{theorem}

\begin{remark}
\label{rem:EVERYONE}
Theorem~\ref{thm:main}, as well as the other results stated in this article also hold in dimension $d=1$, where they are usually immediate consequences of known results in~\cite{ABBS12,ABK13}. In this case, $\sigma$ is the measure $\delta_1+\delta_{-1}$ on the sphere $\S^{0} = \{1,-1\}$. In other words, in all dimension $\sigma$ is the Haar measure on $\S^{d-1}$.
\end{remark}

Before proceeding with the proof, we make several comments regarding Theorem~\ref{thm:main} and its meaning. We first observe that, through classical Poisson computations, the extremal point process can be constructed in the following fashion. Conditionally on the branching Brownian motion, we draw an independent Poisson point process $(\chi_i)_{i \in \N}$ with intensity $\gamma \pi^{-\alpha_d/2} \sqrt{2}e^{-\sqrt{2}y} \dd y$, i.i.d. random variables $(\theta_i)_{i \in \N}$ in $\S^{d-1}$ with law $D_\infty(\theta) \sigma(\d \theta) / D_\infty(\S^{d-1})$ and i.i.d. point processes $(\calD^{(i)})_{i \in \N}$ with the same law as $\mathcal{D}$. Then
\begin{equation}
  \label{eqn:alternativeFormula}
  \mathcal{E}_\infty = \sum_{i = 1}^\infty \sum_{r \in \calD^{(i)}} \delta_{\left(\theta_i,\,\tfrac{1}{\sqrt{2}}\log D_\infty(\S^{d-1}) + \chi_i + r\right)}\,.
\end{equation}
In other words, the extremal point process is constructed from an exponential Poisson point process for the description of the norms, shifted by $\frac{1}{\sqrt{2}} \log D_\infty(\S^{d-1})$, so that to each atom is associated independently an angle sampled proportionally to $D_\infty(\theta)\sigma(\dd \theta)$ and an i.i.d. decoration, whose law does not depend on the dimension.

Additionally, the following Lalley-Sellke type result (Proposition~\ref{prop:Lalley and Sellke}) can be deduced from Theorem  \ref{thm:main}. Define $\mathcal{F}_L=\sigma(X_u(s), u \in \mathcal{N}_L, s\leq L)$, and note that,  almost surely, there exists a unique particle $u \in \mathcal N_t$ such that $R^{(u)}_t = R^*_t$. We  write $\theta^*_t:= \theta^{(u)}_t \in \S^{d-1}$.
\begin{proposition}\label{prop:Lalley and Sellke}
For any measurable set $A\subseteq \S^{d-1}$ such that $\sigma(\partial A) = 0$, we have
\begin{equation}\label{E:Lalley Sellke angle}
\lim_{L\to \infty} \lim_{t \to \infty} \P\left(  \theta^*_t \in A \mid  \cF_L \right) = D_\infty(A)/D_\infty(\S^{d-1}) \; \text{ almost surely,}
\end{equation}
and for all $y \in \R$, we have
\begin{equation}\label{E:Lalley Sellke modulus}
\lim_{L\to \infty} \lim_{t\to\infty} \P( R_t^* \leq m_t^{(d)}+y \given \cF_L ) =  \exp\left(- \gamma\pi^{-\alpha_d/2} D_\infty(\S^{d-1}) e^{-y\sqrt{2}}\right) \; \text{ almost surely.}
\end{equation}
\end{proposition}
Equation \eqref{E:Lalley Sellke angle} which gives the asymptotic direction of the maximal displacement was conjectured in~\cite{StBeMa2020}.  Furthermore, we record Corollary~\ref{cor:cvd}, which is an integrated version of the second statement \eqref{E:Lalley Sellke modulus} and shows that the limit distribution $R_t^* - m_t^{(d)}$ can be expressed in terms of the Laplace transform of $D_\infty(\S^{d-1})$. This is very close to the main result in~\cite{KLZ21} (see Corollary~\ref{cor:equiv-of-deriv-mtgs} for more on how they differ).

Our next comment is that, as in the one-dimensional case, the Poisson point process $\{(\xi_i, \theta_i)\}_{i \in \N}$ corresponds to the so-called ``clan-leaders'', i.e., the particles with maximal displacement in their immediate family. More precisely, for any $r\in(0,t)$ and $u \in \cN_t$ we let $[u]_r = \{ v\in \mathcal{N}_t : u \wedge v \ge t-r \}$, where $u \wedge v$ is the time of the most recent common ancestor of $u$ and $v$. In other words, $[u]_r $ is the set of particles that have branched off $u$ at most $r$ units of time prior to $t$ (the immediate family of $u$). We say that $u\in \cN_t$ is an $r$-{\it clan-leader} if $u$  is the particle which is the furthest away from the origin among $[u]_r $, and we write $\Gamma_t(r)\subset \cN_t$ for the set of all $r$-clan leaders at time $t$.
Let $r(t)$ be any function such that $r(t)\to  \infty$ but $r(t) = o (t)$.
Then our proof of Theorem~\ref{thm:main} can be extended to show that
\[
\mathcal L_t := \sum_{u \in \Gamma_t(r(t)) } \delta_{(\theta^{(u)}_t , R^{(u)}_t- m_t^{(d)})} \, \rightarrow \sum_i\delta_{(\theta_i,\xi_i)} \, \text{in distribution as $t\to \infty$}\,,
\]
where the limit is the Poisson point process on $\S^{d-1} \times \R$ from the statement of the theorem.
Moreover, if for $v \in \calN_t$ we let
\[
\calD_{t,r(t)}(v) := \sum_{w \in [v]_{r(t)}} \delta_{\big(\theta_t(w)-\theta_t(v), R_t(w) - R_t(v)\big)} \,,
\]
and define
\[
  \hat{\mathcal{E}}_t := \sum_{v^* \in \Gamma_t(r(t))} \delta_{\big(\theta_t(v^*), R_t(v^*) - m_t, \calD_{t,r(t)}(v^*) \big)}\,,
\]
then, similarly to the result of Biskup and Louidor for the 2-dimensional Gaussian free field \cite{BiL18}, we have 
\begin{equation}\label{equ:conv of the clan-leaders}
  \hat{\mathcal{E}}_t \rightarrow \hat{\mathcal{E}} \,, \quad
  \text{ where } \quad \hat{\mathcal{E}} := \sum_{i=1}^\infty\delta_{(\theta_i, \xi_i, \hat{\calD}^{(i)})}  \ \ \text{ and } \ \ \hat{\calD}^{(i)} := \sum_{r\in \calD^{(i)}} \delta_{(0,r)} \,,
\end{equation}
and $\theta_i, \xi_i,$ and $\calD^{(i)}$ are as in Theorem~\ref{thm:main}. Here, the convergence is in the weak sense for point measures on $\S^{d-1} \times \R \times \mathbb{M}$, with $\mathbb{M}$ denoting the space of Radon measures on $\R^d \times \R$ endowed with the vague topology.

Although we have not done so in the present work, extending the convergence in Theorem~\ref{thm:main} to carry the genealogical information as in \cite{BoHa17,Mal18} would not add any extra complications.

A key step in proving Theorem~\ref{thm:main}  will be to be able to use the convergence in distribution of the maximal displacement proved in~\cite{KLZ21}. However,
 there, the analogue of the derivative martingale is not given by $D_\infty(\S^{d-1})$ but rather by a certain random variable $Z_\infty$. We thus need to understand the relation between $Z_\infty$ and the measure $D_\infty(\cdot).$
Let
\[
    \cN_t^{\win} := \{u \in \cN_t : R_t^{(u)} \in \sqrt{2}t - [t^{1/6}, t^{2/3}]\}\,,
\]
and, recalling that $\alpha_d := (d-1)/2$, let
\[
\fM_t^{(u)} := (R_t^{(u)})^{-\alpha_d} (\sqrt{2}t - R_t^{(u)}) e^{-(\sqrt{2}t - R_t^{(u)}) \sqrt{2}}\,.
\]
In \cite{KLZ21}, a BBM started from the origin was considered, and the variable $Z_\infty$ was defined in as the limit in distribution of
\[
Z_t := \sum_{u \in \cN_t^{\win}} \fM_t^{(u)} \,.
\]
We show the following.
\begin{theorem}\label{lem:equiv-of-deriv-mtgs}
Let $f: \S^{d-1} \to \R$ be a  continuous  function and $x \in \R^d$, we have
\begin{align}
    \lim_{L \to \infty} (2\pi)^{\alpha_d/2}\!\!\!\! \sum_{u \in \calN_L^{\win}} f(\theta_L^{(u)})
    \fM_L^{(u)}
    =  \langle  D_{\infty}, f \rangle \,
    \label{eqn:equiv-of-deriv-mtgs}
    \quad \text{ in $\P_x$-probability.}
\end{align}
\end{theorem}

In~\cite[Remark~1.2]{KLZ21}, a formal argument was made for the distributional equivalence of $D_{\infty}(\S^{d-1})$ and a positive constant times $Z_{\infty}$;  the statement above is much stronger, showing in particular that the random variable $Z_\infty$ is a measurable function of the branching Brownian motion. More precisely, set $\hat Z_{\infty}  := (2\pi)^{-\alpha_d/2}  D_{\infty}(\S^{d-1}) $. Then the following holds.
\begin{corollary}\label{cor:equiv-of-deriv-mtgs}
$Z_t$ converges to $\hat Z_{\infty}$ in $\P_x$-probability. In particular, taking $x = 0$, we have $Z_\infty\stackrel{d}{=}\hat Z_\infty$.
\end{corollary}
See also Corollary~\ref{cor:cvd} for the convergence in distribution of the maximal displacement of a BBM started from an arbitrary point $x \in \R^d$--- this extends \cite[Theorem~1]{KLZ21}.
Theorem~\ref{thm:main} is an immediate consequence of the study of the convergence of the Laplace transform of $\mathcal{E}_t$ as $t \to \infty$ and the identification of the limit with the Laplace transform of $\mathcal{E}_\infty$.
\begin{proposition}
\label{prop:laplaceFunctional}
Let $\phi: \S^{d-1}\times \R \to \R_+$ be a continuous, compactly-supported, non-negative function.
For all $x \in \R^d$, we have
\[
  \lim_{t \to \infty} \E_x\Big[ \exp\Big( - \sum_{u \in \cN_t} \phi(\theta_t^{(u)}, R_t^{(u)} - m_t^{(d)}) \Big) \Big] =
  \E_x\Big[ \exp\Big( - \fC_d \int_{\S^{d-1}} C(\phi_{\theta}) D_{\infty}(\theta) \sigma(\d\theta) \Big) \Big]\,,
\]
where
\begin{align}
\fC_d := \sqrt{\frac{2}{\pi^{1+\alpha_d}}}\,, \label{def:fcd}
\end{align}
$\phi_{\theta}(\cdot) := \phi(\theta,\cdot)$,  and $C(\phi_{\theta})$ is defined in~\eqref{def:C(phi)}.
\end{proposition}

The structure of the paper is as follows.  In Section~\ref{sec:lemmas}, we describe several technical results. These include: the description by~\cite{KLZ21} of the trajectories of the norms of extremal particles (those reaching height within constant distance of $m_t^{(d)}$);  a simple but key stability result for the process of the angles of extremal particles (Proposition~\ref{prop:sector-result}); and a recollection of convergence results for the F-KPP equation that will be utilized throughout the rest of the paper.

In Section~\ref{sec:proof-equiv-deriv-mtgs}, we prove Theorem~\ref{lem:equiv-of-deriv-mtgs} by carefully examining the contribution of each  particle $v\in \cN_L$ to the integral $\crochet{D_L, f}$ via Laplace's saddle point method (Lemma~\ref{lem:analysisTransform}). The particles that contribute turn out to be those in $N_L^{\win}$ (Lemma~\ref{lem:non-window}), and their contributions can be matched with the terms in the sum defining $Z_L$. The $\P_x$-almost-sure convergence of $\crochet{D_L, f}$  to the right-hand side of~\eqref{eqn:equiv-of-deriv-mtgs} allows us to conclude.

In Section~\ref{sec:proof-laplace-functional}, we prove Proposition~\ref{prop:laplaceFunctional} using a key leading-order tail asymptotic on the Laplace functional (Proposition~\ref{prop:tailEstimate}) in combination with the branching property, as well as Theorem~\ref{lem:equiv-of-deriv-mtgs}. We then prove Theorem~\ref{thm:main} using Proposition~\ref{prop:laplaceFunctional} and the identification of the Laplace transform of $\mathcal{E}_\infty$. We end that section by proving Proposition~\ref{prop:Lalley and Sellke} as a consequence of Theorem~\ref{thm:main}.

Proposition~\ref{prop:tailEstimate} is then proved in Section~\ref{sec:proof-tail-estimate} using information on the trajectories of the norms of the extremal BBM particles and a coupling with one-dimensional BBM similar to the one used in~\cite{KLZ21}. Crucially,  the information on the norm trajectories is enough, due to the angular stability result Proposition~\ref{prop:sector-result}.

\section{Preliminaries}
\label{sec:lemmas}
\subsection{Notation for asymptotics}
\label{subsec:asymp-notation}
For functions $f(t)$ and $g(t)$, we  write $f \sim g$ to denote the relation $f/g \to 1$ as $t \to \infty$. When needed, we emphasize the dependence on $t$ by writing
$f\sim_t g$.  We  write $f \lesssim g$  to mean there exists some constant $C>0$ such that for all $t$ sufficiently large, $f(t) \leq Cg(t) $. We write $f\asymp g$ to mean $f\lesssim g$ and $g \lesssim f$.

In what follows,
we will consider time parameters $t$ and $L$, where $t$ is sent to infinity before $L$. We will also consider a parameter $z \in [L^{1/6},L^{2/3}]$.
For functions $f:=f(t,L,z)$ and $g:=g(t,L,z)$,
we write
$f \usim g$ to denote the relation
\begin{align*}
    \lim_{L \to \infty}
    \liminf_{t\to\infty}
    \inf_{z \in [L^{1/6}, L^{2/3}]}
    \frac{f}{g} =
    \lim_{L \to \infty}  \limsup_{t\to\infty}
    \sup_{z \in [L^{1/6}, L^{2/3}]}
    \frac{f}{g} =
    1 \,.
\end{align*}
We  write
$f= o_u(g)$ if
\[
    \limsup_{L \to \infty} \limsup_{t \to \infty} \sup_{z \in [L^{1/6}, L^{2/3}]}  \frac{f}{g} = 0 \,.
\]
When functions have no dependency in the variable $z$, we still write $\sim_{(u)}$ and $o_u$ as above, ignoring the $\sup$ and $\inf$ over $z$.

\subsection{Trajectories of the norms of the extremal particles} \label{subsec:trajectories}
A key step towards the convergence result of~\cite{KLZ21} was the following characterization of the trajectories of the norms of particles that reach height $m_t^{(d)}+y$ at time $t$, where  $y\in \R$ is a constant. Let $L$ be a time parameter that is sent to infinity \textit{after} $t$ (so, from the perspective of $t$, $L$ is just a large constant), and let $\ell := \ell(L)$ be any function such that $\ell \in [1, L^{1/6}]$ and $\ell \to \infty$ as $L \to \infty$. Then, with probability $1- o_u(1)$, any particle $v \in \cN_{t-\ell}$ that produces a descendent $u \in \cN_t$ such that $R_t^{(u)} > m_t^{(d)}+y$ did the following:
\begin{enumerate}
    \item $R_L^{(v)} \in I_L^{\win} :=  \sqrt{2}L - [L^{1/6}, L^{2/3}]$ ;
    \item for $s \in [L,t-\ell]$, $R_s^{(v)}$ was bounded above by $m_t^{(d)} s/t +y$ and below by some explicit function ; and
    \item $R_{t-\ell}^{(v)} \in m_t^{(d)}(t-\ell)/t +y - [\ell^{1/3}, \ell^{2/3}]$\,.
\end{enumerate}
In words, the  norm of $v$ at a constant order time from the beginning and from the end lies in a small window; in between these times, the norm of $v$ stays in a sufficiently tight corridor. See Figure~\ref{fig:intro-lb} for a depiction of such a trajectory.
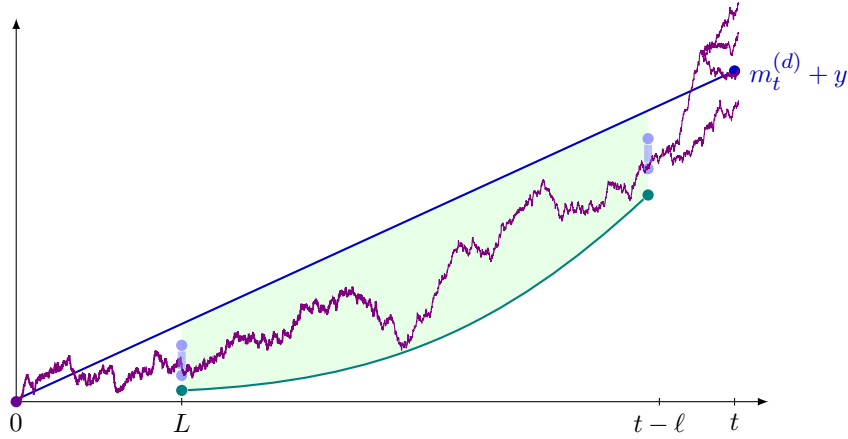
\begin{figure}
     \begin{tikzpicture}[>=latex,font=\small]
 \draw[->] (0, 0) -- (10, 0);
  \draw[->] (0, 0) -- (0, 5.1);
   \coordinate  (b0) at (9.55,0);
   \coordinate  (L) at (2.2,0);
  \fill[color=green!10] (2.2,0.15) to[bend right=20] (8.4,2.75) -- (8.4,3.88) -- (2.2,1.05);
   \node[circle,fill=blue!45!purple,inner sep=1.5pt] (o) at (0, 0) {};
   \node[circle,fill=blue!75!black,label={[blue!75!black]right:$m_t^{(d)}+y$},inner sep=1.5pt] (b) at ($(b0)+(0,4.4)$) {};
   \draw[blue!75!black, thick] ($(o)+(0.05,0.05)$) -- (b);
   \node[circle, fill=blue!50, inner sep=1.5pt, opacity=0.75] (y1) at ( $(b)-(1.15,0.9)$) {};
   \node[circle, fill=blue!50, inner sep=1.5pt, opacity=0.75] (y2) at ( $(y1)-(0,0.4)$) {};
   \node[circle, fill=blue!50, inner sep=1.5pt, opacity=0.75] (x1) at ( $(L)+(0,0.75)$) {};
   \node[circle, fill=blue!50, inner sep=1.5pt, opacity=0.75] (x2) at ( $(x1)-(0,0.4)$) {};
  \node[circle, fill=green!50!blue, inner sep=1.5pt] (ba1) at ( $(x2)-(0,0.2)$ ) {};
  \node[circle, fill=green!50!blue, inner sep=1.5pt] (ba2) at ( $(y2)-(0,0.35)$ ) {};
  \draw[green!50!blue, thick] (ba1) to[bend right=20] (ba2);
  \node (fig1) at (4.82,2.35) {
  \includegraphics[width=0.67\textwidth]{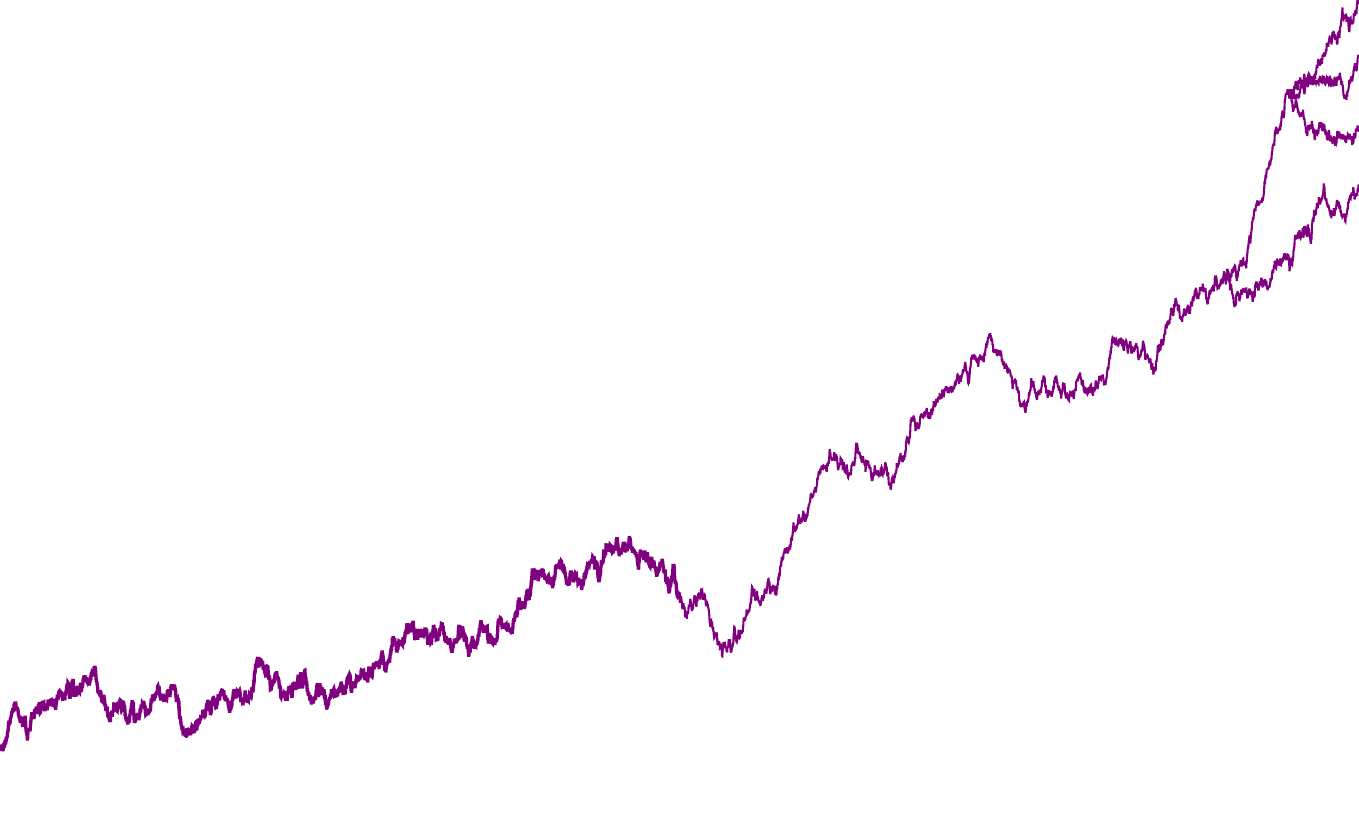}};
  \draw[blue!50, line width=3pt, opacity=0.5] (y1) -- (y2);
    \draw[blue!50, line width=3pt, opacity=0.5] (x1) -- (x2);
  \node[below] at (0,-0.05) {$0$};
  \draw ($(L)+(0,0.05)$) -- ($(L)+(0,-0.05)$) node[below] {$L$};
  \draw ($(b0)+(-1,0.05)$) -- ($(b0)+(-1,-0.05)$) node[below] {$ t - \ell$};
  \draw ($(b0)+(0,0.05)$) -- ($(b0)+(0,-0.05)$) node[below] {$ t$};
  \end{tikzpicture}
     \caption{Trajectory of the norm of a typical particle considered in $\cE_t$: at time $L$ are at height in $I_L^\win$, stay in the shaded (green) region up to time $t-\ell$, at time $t-\ell$ are located in another window, and then produce a descendant that reaches $m_t^{(d)}+y$ at time $t$.}
     \label{fig:intro-lb}
 \end{figure}

This characterization of the extremal trajectories will be key for the proof of Proposition~\ref{prop:laplaceFunctional}.
More precisely, Proposition~\ref{prop:laplaceFunctional} follows quickly from the tail estimate Proposition~\ref{prop:tailEstimate}, the proof of which completely relies on the above trajectory characterization. This proof is given in Section~\ref{sec:proof-tail-estimate}, where the trajectory characterization is given in full detail along with genealogical information: see Propositions~\ref{prop:klz21-1stmom-results} and~\ref{prop:klz21-2ndmom}. Prior to Section~\ref{sec:proof-tail-estimate}, we will use item 1 above  multiple times, and so we state it precisely below.
\begin{proposition}[{\cite[Theorem~3.1]{KLZ21}}] \label{prop:window-time-L}
For any $y \in \R$, we have
\begin{align*}
        \lim_{L\to \infty}
        \limsup_{t\to\infty}
        \P
        \Big(\exists v \in \cN_t :
             ~R_L^{(v)} \not \in I_L^{\win}\,,\,
            R_t^{(v)} > m_t^{(d)} + y
            \Big)
    = 0 \,.
\end{align*}
\end{proposition}

\subsection{Many-to-one lemma and multidimensional Brownian motions}
\label{subsec:bessel}
Many-to-few lemmas are ubiquitous tools in the study of spatial branching processes. They connect the moments of additive functionals of the branching process with estimates related to a typical trajectory. In this article, we use a simple version of the many-to-one lemma that relates the mean of an additive functional of the branching Brownian motion with a Brownian motion estimate. We refer to~\cite{HaRo17} for the description of the general settings.
\begin{lemma}[Many-to-one lemma]
Fix $ d \in \N$, and let $X_{\cdot}$ denote a $d$-dimensional Brownian motion.
For any $T\geq 0$ and $y \in \R$, and for any non-negative measurable function $f:C^d[0,T] \to \R$,  we have
\label{lem:many-to-one}
\begin{align}
    \E_x\bigg[ \sum_{v \in \cN_T} f\big((X_s^{(v)})_{s \leq T}\big)\bigg]
    =
    e^T \E_x \Big[f\big((X_s)_{s \leq T}\big)\Big]\,.
    \label{eqn:many-to-one}
\end{align}
Here,  $C^d[0,T]$ denotes the set of continuous functions from $[0,T]$ to $\R^d$.
\end{lemma}

Recall that the norm of a $d$-dimensional Brownian motion is a $d$-dimensional Bessel process.
In particular, $\{R_s^{(v)}\}_{s>0 , v\in \cN_s}$ is a branching Bessel process.
Throughout, we will write $R_.$ to denote the process given by the norm of standard $d$-dimensional Brownian motion, and we will write $W_.$ to denote a standard Wiener process. When $R_0 >0$ and $d \geq 2$, we have the following SDE  (see~\cite[Chapter~XI]{ReYo99} for a treatment of Bessel processes):
\begin{align}
    \d R_t = \frac{\alpha_d}{R_t} \d t + \d W_t,
    \label{eqn:bessel_SDE}
\end{align}
where we recall that $\alpha_d := (d-1)/2$.

We will also use the fact that $\|X_L^{(v)}\|  \overset{(d)}{=} L^{1/2} \chi_d$ for $X_0^{(v)} = 0$, where $\chi_d$ is a chi random variable with $d$ degrees of freedom.
Letting $p_L^R(\cdot,\cdot)$ denote the transition density of a $d$-dimensional Bessel process at time $L$ and $p^{\chi_d}$ denote the density of $\chi_d$, we have
\begin{align}
    p_L^R(0, r) &= L^{\frac{1}2} p^{\chi_d} \big( L^{-\frac{1}2}r \big) = c_d L^{-d/2} r^{d-1} e^{- \frac{r^2}{2L}} \,.
    \label{eqn:bessel-density}
\end{align}
In particular, by integration by parts, there exists $C_d > 0$ such that for all $a,L>0$ we have
\begin{equation}
  \label{eqn:tailBessel}
  \P(R_L \geq a) = \frac{c_d}{L^{1/2}}\int_a^\infty \left( \frac{r}{L^{1/2}} \right)^{d-1} e^{-r^2/2L} \dd r \leq C_d \left( \frac{a}{L^{1/2}} \right)^{d-2} e^{-a^2/2L}.
\end{equation}

It is worth noting that considering the polar decomposition of a Brownian motion $B$ in $\R^d$ as the diffusion $((R_t,\theta_t), t \geq 0)$ on $\R_+ \times \S^{d-1}$ gives $(R_t,t\geq 0)$ as a $d$-dimensional Bessel process, and conditionally on the latter, gives $(\theta_t, t\geq 0)$ as a time-inhomogeneous Brownian motion on the sphere with diffusion constant ${1}/{R_t^2}$ at time $t$. In particular, $\theta_t$ converges in law as $t \to \infty$ to the uniform distribution on the sphere. However, note that conditionally on $\{ R_t \geq \epsilon t, t \geq 0\}$, $\theta_t$ converges almost surely to a random point of the sphere.

\subsection{Stability of the angular process}
As the radial part of the typical trajectory of a particle at distance $m_t^{(d)}$ at time $t$ has grown linearly over time, we deduce from the observation above that its angular part $\theta^{(u)}_t$ should be converging, and in particular be close to $\theta^{(u)}_s$ for $s$ large enough. The main result of this section confirms this heuristic by  stating that the direction of extremal particles at time $t$ are very close to the direction of their ancestors at time $L$ with high probability. This proves that the direction of all individuals in the same clan is identical.
Before stating the result, we introduce  the following notation: for any $s < t$ and $v \in N_s$, define
\[
\calN_{t-s}^v := \{ u \in \calN_t : \text{$u$ is a descendant of $v$}\}\,.
\]

\begin{proposition}\label{prop:sector-result}
Fix $ y \in \R$. For all $L$ large enough, we have
\begin{align}
  \limsup_{t \to \infty}
  \P \big( \exists v\in \cN_L^{\win} \, , \, u \in \cN_{t-L}^v : \| \theta_L^{(v)} - \theta_t^{(u)}\| \geq 2  L^{-1/12} \,,\, R_t^{(u)}  > m_t^{(d)}+ y\big)
  \leq e^{-L^{5/6}}. 
    \label{eqn:prop-sector-result}
\end{align}
\end{proposition}

Proposition \ref{prop:sector-result} is an immediate consequence of the two following claims.

\begin{claim}\label{claim:L5/6-bad}
For all $L$ sufficiently large and $t$ sufficiently large compared to $L$,
\begin{align}
    \P\Big( \exists v\in \calN_L^{\win} \, , \, u \in \cN_{t-L}^v :   \|X_t^{(u)} - X_L^{(u)}\| \geq m_{t-L}^{(d)} + L^{5/6} \Big)
    \leq e^{-L^{5/6}}\,.
    \label{eqn:claim-L5/6-bad}
\end{align}
\end{claim}

\begin{proof}
By the Markov inequality, we have
\begin{align*}
  &\P\Big( \exists v\in \calN_L^{\win} , \, u \in \cN_{t-L}^v :   \|X_t^{(u)} - X_L^{(u)}\|
  \geq m_{t-L}^{(d)} + L^{5/6} \Big) \\
  &\leq \E\bigg[ \sum_{v \in \mathcal{N}_L^{\win}} \ind{\exists \, u \in \cN_{t-L}^v :   \|X_t^{(u)} - X_L^{(u)}\| \geq m_{t-L}^{(d)} + L^{5/6}} \bigg]\,.
\end{align*}
Observe that by the Markov property and the shift-invariance of the $d$-dimensional Brownian motion, the process $(X^{(u)}_{L+s} - X^{(v)}_L~,~u \in \mathcal{N}^{(v)}_{L+s}~,~ s \geq 0)$ is a ($d$-dimensional) branching Brownian motion started from $0$, independent of $(X^{(v)}_s, v \in \mathcal{N}_s, s \leq L)$. Therefore, by the many-to-one lemma, the right-hand side of the previous display is bounded above by
\begin{multline}
  \E\left[ \# \mathcal{N}_L^\win \right] \P(R^\ast_{t-L} \geq m^{(d)}_{t-L}+L^{5/6}) \\
  \leq e^L \P(R_L - \sqrt{2} L \in [-L^{2/3},-L^{1/6}])\P(R^\ast_{t-L} \geq m^{(d)}_{t-L}+L^{5/6})\,.
  \label{eqn:thebound1}
\end{multline}
Using \eqref{eqn:tailBessel}, we have
\[
   e^L \P(R_L - \sqrt{2} L \in [-L^{2/3},-L^{1/6}]) \leq C_d L^{d/2 - 1} e^{L- \frac{(\sqrt{2}L - L^{2/3})^2}{2L}} \leq C_d L^{d/2-1} e^{\sqrt{2}L^{2/3}}\,.
\]
Additionally, applying~\cite[Equation~1.2]{Mallein15}, there exists $K_d>0$ such that for all $t,L > 0$,
\begin{equation*}
  \P(R^\ast_{t-L} \geq m^{(d)}_{t-L}+L^{5/6}) \leq K_d e^{-\sqrt{2} L^{5/6}}.
\end{equation*}
As a consequence, \eqref{eqn:thebound1} implies that for all $L$ large enough,
\[
  \limsup_{t \to \infty} \P\Big( \exists v\in \calN_L^{\win} \, , \, u \in \cN_{t-L}^v :   \|X_t^{(u)} - X_L^{(u)}\| \geq m_{t-L}^{(d)} + L^{5/6} \Big) \leq e^{-L^{5/6}}\,,
\]
completing the proof.
\end{proof}

Claim \ref{claim:L5/6-bad} states that with high probability, extremal particles at time $t$ stay within distance $m_{t-L} + L^{5/6}$ from their ancestor at time $L$. We now use simple geometry to conclude that in this case, the direction of extremal particles have to stay close to the direction of their ancestor at time $L$, as illustrated in Figure~\ref{fig:angle-drawing}.

\begin{claim}
\label{claim:otherBadAngle}
Let $L > 0$ and $x \in \R^d$ such that $L - L^{2/3} \leq \|x\| \leq L - L^{1/6}$. For all $L$ large enough, we have
\[
  \limsup_{R \to \infty} \sup_{z \in B(x, R+L^{5/6}) \setminus B(0, R+L)} \left\| \frac{x}{\|x\|} - \frac{z}{\|z\|} \right\| \leq \sqrt{2} L^{-1/12}\,.
\]
\end{claim}

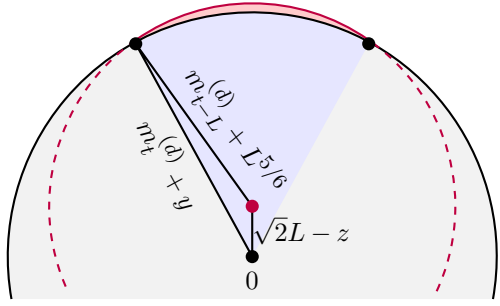
\begin{figure}
    \centering
    \begin{tikzpicture}[font=\small]
  \pgfmathsetmacro{\degA}{-10}
  \pgfmathsetmacro{\radA}{3.25pt}
  \pgfmathsetmacro{\radB}{2.7pt}
  \pgfmathsetmacro{\degB}{55}
  \coordinate (a) at (0,0);
  \coordinate (b) at (0,.67);
  \coordinate (x) at ($(\degB:\radB)+(b)+(0,-0.05)$);
  \coordinate (y) at ($(180-\degB:\radB)+(b)+(0,-0.05)$);
  \draw[thick,purple,fill=red!20] ([shift=(\degB:\radB)]b)  arc (\degB:180-\degB:\radB);
  \fill[gray!10] ([shift=(\degA:\radA)]a) arc (\degA:180-\degA:\radA);
  \fill[blue!10] (a)--(x) to[bend right=28] (y);
  \draw[thick,black] ([shift=(\degA:\radA)]a) arc (\degA:180-\degA:\radA);
  \draw[thick,dashed,purple] ([shift=(-\degB*.45:\radB)]b)  arc (-\degB*.45:\degB:\radB);
  \draw[thick,dashed,purple] ([shift=(180-\degB:\radB)]b)  arc (180-\degB:180+\degB*0.45:\radB);
  \draw[black,thick] (a)--(b) --(y) node [pos=0.35,sloped,above,xslant=-0.06]{$m_{t-L}^{(d)}+L^{5/6}$} --(a) node [pos=0.5, sloped, below]{$m_t^{(d)}+y$};
  \node[circle, fill=black, inner sep=1.75pt, label=below:$0$] at (a) {};
  \node[circle, fill=purple, inner sep=1.75pt] at (b) {};
  \node[circle, fill=black, inner sep=1.75pt] at (x) {};
  \node[circle, fill=black, inner sep=1.75pt] at (y) {};
  \node at ($(a)+(0.65,.35)$) {$\sqrt2 L - z$};
\end{tikzpicture}
    \centering
    \caption{The marked cone corresponds to the domain to which $\theta_L^{(v)} - \theta_t^{(v)}$ belongs with high probability if $\|X_t^{(v)}\| \geq m_t^{(d)}+y$ while $v \in \mathcal{N}^{\win}_L$.}
    \label{fig:angle-drawing}
\end{figure}

\begin{proof}
For  $z \in B(x,R+L^{5/6})\setminus B(0,R+L)$, straightforward computations yield that
\begin{align*}
  \frac{x}{\|x\|} \cdot \frac{z}{\|z\|} &\geq \frac{\|z\|^2 - \|x-z\|^2}{2\|x\|\|z\|} \geq \frac{(R+L)^2 - (R + L^{5/6})^2}{2\|x\| \|z\|} \geq \frac{(2R + L + L^{5/6})(L- L^{5/6})}{2(L - L^{1/6})(R+2L)}\\
  &\geq \frac{1-L^{-1/6}}{1-L^{-5/6}} \frac{2R+L+L^{5/6}}{2R+4L}\,.
\end{align*}
We observe that for all $L$ large enough,
\[
  \liminf_{R\to \infty} \frac{1-L^{-1/6}}{1-L^{-5/6}} \frac{2R+L+L^{5/6}}{2R+4L}\geq 1 - L^{-1/6}\,.
\]
As a result, using that
\[
  \left\| \frac{x}{\|x\|} - \frac{z}{\|z\|} \right\| 
  = \sqrt{2}\sqrt{1 -  \frac{x}{\|x\|} \cdot \frac{z}{\|z\|}}\,,
\]
we conclude that for all $L$ large enough,
\[
  \limsup_{R \to \infty} \sup_{z \in B(x,R+L^{5/6}) \setminus B(0,R+L)} \left\| \frac{x}{\|x\|} - \frac{z}{\|z\|} \right\| \leq \sqrt{2} L^{-1/12}\,.\qedhere
\]
\end{proof}

We can now complete the proof of Proposition~\ref{prop:sector-result}.

\begin{proof}[Proof of Proposition~\ref{prop:sector-result}]
From Claim~\ref{claim:otherBadAngle}, we have that
\begin{align*}
  &  \P \big(\exists v \in \cN_L^{\win}, u\in \cN_{t-L}^{v} : \| \theta_L^{(v)} - \theta_t^{(u)}\| \geq 2  L^{-1/12} , \, \, \, R_t^{(u)}  > m_t^{(d)}+ y\big)\\
  &\leq  \P\left( \exists v \in \cN_L^{\win}, u\in \cN_{t-L}^{v} : \|X_t^{(u)} - X_L^{(v)}\| \geq m_{t-L}^{(d)} + L^{5/6} \right) \,.
  \end{align*}
The bound in Claim~\ref{claim:L5/6-bad} finishes the proof.
\end{proof}

\subsection{Convergence results for the  F-KPP Equation}
\label{subsec:fkpp}
Branching Brownian motion is connected to the F-KPP reaction-diffusion equation. More precisely, McKean's representation connects multiplicative functionals of the one-dimensional BBM to solutions of the F-KPP equation:
\begin{proposition}[\cite{McKean75}] \label{prop:mckean}
Let $f: \R \to [0,1]$, and let $\{W_s^{(v)}\}_{s \geq 0, v \in N_s}$ denote a one-dimensional BBM.
Then, for any $w \in \R$, the F-KPP equation
\[
\partial_t u = \frac{1}{2}\partial_x^2 u - u + u^2
\]
with initial conditions $u(0,x) = 1- f(x)$ is solved by
\begin{align}
u(t, x) := \E\Big[ 1- \prod_{v \in N_{t}} f(x-W_{t}^{(v)}) \Big]\,. \label{eqn:full-mckean}
\end{align}
\end{proposition}

We will appeal to the following F-KPP convergence result several times.
\begin{proposition}[{\cite[Proposition~3.2, Lemma~4.6, Lemma~4.8]{ABK13}}] \label{prop:abk-tailconvergence}
Let $u(t,x)$ solve the F-KPP equation with initial condition $g(x) \in [0,1]$ satisfying
\[
\sup \{ y : g(y) > 0 \} < \infty\,.
\]
Then there exists a positive, finite constant $C_g$ depending only on $g$ such that for any constant $c \in \R$, we have
\begin{align*}
    C_g e^{\sqrt{2}c} &= \lim_{\ell \to \infty} \int_{0}^{\infty} we^{\sqrt{2}w} u(\ell, \sqrt{2}\ell +w+c) \d w
    =\lim_{\ell \to \infty} \int_{\ell^{1/3}}^{\ell^{2/3}} we^{\sqrt{2}w} u(\ell, \sqrt{2}\ell +w+c) \d w \,.
\end{align*}
\end{proposition}
We make explicit how the Proposition \ref{prop:abk-tailconvergence}
follows from~\cite{ABK13}. Equation~$(3.3)$ of Proposition~$3.2$
of~\cite{ABK13}
tells us that the following limit exists:
\begin{align}
    C_g := \lim_{\ell \to \infty}  \int_{0}^{\infty} we^{\sqrt{2}w} u(\ell, \sqrt{2}\ell +w) \d w ~\in (0, \infty)\,.
    \label{def:Cf}
\end{align}
Lemma~$4.6$ of~\cite{ABK13} tells us that $C_g$ is equal to the limit as $\ell$ tends to infinity of the above integral taken only over $[\ell^{1/2-\delta}, \ell^{1/2+\delta}]$, for any $\delta \in~(0,1/2)$. The shift by $c$ resulting in the $e^{\sqrt{2}c}$ pre-factor is now a direct consequence of Lemma~$4.8$ of~\cite{ABK13}.

An important application of McKean's theorem is to the case $u(0,x) = \ind{x <0}$. Letting $W_{\ell}^* := \max_{v \in N_{\ell}} W_{\ell}^{(v)}$,
 Proposition~\ref{prop:mckean} states that
$\P(W_\ell^* >x)$ solves the F-KPP equation, whence Proposition~\ref{prop:abk-tailconvergence} yields the following positive constant:
\begin{align}
    \gamma := C_{\ind{x <0}}=  \lim_{\ell \to \infty} \int_0^{\infty} w e^{\sqrt{2}w} \P(W_{\ell}^*>\sqrt{2}\ell+ w) \d w \,. \label{def:gamma}
\end{align}
This constant appears in the limiting law of the re-centered maximum of BBM in every dimension.
In~\cite{LS87}, Lalley and Sellke showed that there exists some positive constant $C>0$
\[
\lim_{t\to\infty} \P(W_t^* -m_t^{(1)} \leq y) =
\E\big[\exp \big(-C Z^{(1)} e^{-y\sqrt{2}} \big)\big]\,,
\]
where $Z^{(1)}$ denotes the derivative martingale from one-dimensional BBM. (The constant $C$ was identified as $C =  \gamma \sqrt{2/\pi}$, see
\cite{ABK13}.) In~\cite{KLZ21}, the main result (Theorem 1) states that
\[
\P(R_t^* -m_t^{(d)} \leq y) = \E\Big[\exp \Big(- \sqrt{\tfrac{2^{1+\alpha_d}}{\pi}} \gamma Z_{\infty} e^{-y\sqrt{2}} \Big)\Big] \,.
\]
In contrast to the above, in~\cite[Theorem~1]{KLZ21}, the constant in front of $Z_{\infty}$ is not written out explicitly. Instead it is just called ``$\gamma^*$''. The following expression for $\gamma^*$ is given by~\cite[Lemma~$5.1$ and Proposition~$5.4$]{KLZ21}:
\begin{align}
\gamma^* = \sqrt{\tfrac{2^{1+\alpha_d}}{\pi}} \lim_{\ell \to \infty} \int_{\ell^{1/3}}^{\ell^{2/3}} w e^{\sqrt{2}w} \P(W_{\ell}^*>\sqrt{2} \ell + w) \d w \, = \sqrt{\tfrac{2^{1+\alpha_d}}{\pi}} \gamma . \label{def:gamma-star}
\end{align}
This is shown to be a positive constant in~\cite[Section~5.4]{KLZ21} using purely probabilistic methods, without reference to the F-KPP equation. Proposition~\ref{prop:abk-tailconvergence} tells us that the above limit actually equals $\gamma$.

We will often consider the solution $u_\phi$ to the F-KPP equation with initial conditions $u(0,x) = 1- e^{-\phi(-x)}$, for a non-negative, compactly supported function $\phi : \R \to \R_{+}$.
Proposition~\ref{prop:mckean} states that, for any $w \in \R$,
\begin{align}
u_\phi(\ell, \sqrt{2}\ell+w) =\E\Big[ 1- \exp \Big(-\sum_{u \in \cN_{\ell}} \phi(W_{\ell}^{(u)}- \sqrt{2}\ell-w) \Big)\Big]
\,. \label{eqn:mckean}
\end{align}
Furthermore, Proposition~\ref{prop:abk-tailconvergence} states that
the following limit exists, and is positive and finite:
\begin{align}
    C(\phi) := \lim_{\ell \to \infty}  \int_{0}^{\infty} we^{\sqrt{w}} \E\Big[ 1- \exp \Big(-\sum_{u \in \cN_{\ell}} \phi(W_{\ell}^{(u)}- \sqrt{2}\ell-w) \Big)\Big] \d w  \,.
    \label{def:C(phi)}
\end{align}
The constant $C(\phi)$ appears numerous times in the sequel.

\section{Proof of Theorem~\ref{lem:equiv-of-deriv-mtgs} (equivalence of \texorpdfstring{$Z_{\infty}$}{Zinfty} and \texorpdfstring{$D_{\infty}$}{Dinfty})}
\label{sec:proof-equiv-deriv-mtgs}

We prove Theorem~\ref{lem:equiv-of-deriv-mtgs} by identifying the main contribution to $D_\infty$, and showing that it coincides with the main contribution of $Z_\infty$. Towards this end, in the following lemma, for any $L,y >0$ and $\psi \in \S^{d-1}$, we estimate the contribution of a particle located at $r \psi \in \R^d$ to $\crochet{D_L, f(\cdot)}$.

\begin{lemma}
\label{lem:analysisTransform}
Let $f : \S^{d-1} \to (0,\infty)$ be a continuous function and fix constants $K>0$ and $\epsilon > 0$. Then, uniformly in $\psi \in \S^{d-1}$ and $y \in [\epsilon L, \sqrt{2}L - K \log L]$, we have that
\begin{align}
  \nonumber
&\int_{\S^{d-1}} f(\theta) (\sqrt{2} L - y \psi \cdot \theta) e^{\sqrt{2} (y \psi \cdot \theta - \sqrt{2} L)} \sigma(\d \theta) \\
&\quad \sim_{L \to \infty}
  (2 \pi)^{(d-1)/4} f(\psi) y^{-\alpha_d} (\sqrt{2}L - y) e^{\sqrt{2}(y-\sqrt{2}L)}\,.
  \label{eqn:laplacemeth-final-good}
\end{align}
Furthermore, there exists $C_{d} > 0$ such that for all $L$ large enough and $y \in [\sqrt{2}L-K \log L, \sqrt{2}L+K\log L]$, we have that
\begin{equation}
    \abs{\int_{\S^{d-1}} (\sqrt{2} L -y \psi \cdot \theta) e^{\sqrt{2} (y \psi \cdot \theta - \sqrt{2} L)} \sigma(\d\theta)} \leq C_{d} (\log L) L^{-\alpha_d} e^{\sqrt{2}(y-\sqrt{2}L)}\,.
    \label{eqn:laplacemeth-bad-x}
\end{equation}
\end{lemma}

\begin{proof}
For $\psi \in \S^{d-1}$ and $\delta > 0$, define the subset $B(\psi,\delta) = \{ \theta \in \S^{d-1} : \psi\cdot\theta > \cos(\delta) \} \subset \S^{d-1}$. We first bound
\begin{align*}
  &\left|\int_{\S^{d-1} \setminus B(\psi,\delta)} f(\theta) (\sqrt{2} L - y \psi \cdot \theta) e^{\sqrt{2} y \psi \cdot \theta} \sigma(\d \theta)\right|
  \\
  &\quad \leq~\|f\|_\infty \int_{\S^{d-1} \setminus B(\psi,\delta)} |\sqrt{2} L - y\psi \cdot \theta| e^{\sqrt{2} y \psi \cdot \theta} \sigma(\d \theta)\\
  &\quad =~ \mathrm{vol}(\S^{d-2}) \|f\|_\infty \int_{\delta}^\pi |\sqrt{2} L - y\cos(\phi)| \sin(\phi)^{d-2} e^{\sqrt{2} y\cos(\phi)} \dd \phi \,,
\end{align*}
where we used the change of variables $\theta \mapsto (\phi, \vec{z}) \in [0,\pi] \times \S^{d-2}$, so that  the sphere $\S^{d-1}$ is parametrized by $R_{\psi}(\cos \phi, \vec{z} \sin \phi)$, where $R_\psi$ is a fixed rotation sending $e_1$ to $\psi$.
Therefore, for all $\delta$ small enough that $\cos(\delta) \leq 1 - \delta^2/4$, there exists a constant $K_d>0$ such that for all $L$ large enough and $y \geq \epsilon L$, we have
\begin{equation}
  \label{eqn:farContrib}
  \left|\int_{\S^{d-1} \setminus B(\psi,\delta)} f(\theta) (\sqrt{2} L -y \psi \cdot \theta) e^{\sqrt{2} y \psi \cdot \theta - \sqrt{2}L} \sigma(\d \theta)\right| \leq  K_d  \|f\|_\infty L e^{\sqrt{2}(y - \sqrt{2}L) -  \sqrt{2} \epsilon \delta^2 L/4}\,.
\end{equation}
Note that the right-hand side of~\eqref{eqn:farContrib} is dominated by the right-hand of~\eqref{eqn:laplacemeth-final-good}. In particular, as $L$ becomes large, the mass of $\int_{\S^{d-1}} f(\theta) (\sqrt{2} L - y \psi \cdot \theta) e^{\sqrt{2} y \psi \cdot \theta} \sigma(\d \theta)$ concentrates on $B(\psi,\delta)$--- we show this now.

The continuous function $f$ on the compact space $\S^{d-1}$ is uniformly continuous; hence, for all $\eta> 0$, there is $\delta=\delta(\eta)$ small enough so that
\begin{align}
  \label{eqn:unifContinu}
  &(f(\psi) - \eta) \int_{B(\psi,\delta)} (\sqrt{2} L - y \psi \cdot \theta) e^{\sqrt{2} y \psi \cdot \theta} \sigma(\d \theta)\\
  &\qquad\leq \int_{B(\psi,\delta)} f(\theta) (\sqrt{2} L - y \psi \cdot \theta) e^{\sqrt{2}y \psi \cdot \theta} \sigma(\d \theta)\nonumber\\ &\qquad
  \leq (f(\psi) + \eta) \int_{B(\psi,\delta)} (\sqrt{2} L - y \psi \cdot \theta) e^{\sqrt{2} y \psi \cdot \theta} \sigma(\d \theta) \nonumber\,.
\end{align}
Therefore, to complete the proof, it is enough to compute the asymptotic behaviour for large $y$ and $L$ of
\begin{align*}
  I_{d,\delta}(L,y) &= \int_{B(\psi,\delta)} (\sqrt{2} L - y \psi \cdot \theta) e^{\sqrt{2} y \psi \cdot \theta} \sigma(\d \theta)\\
  &= \mathrm{vol}(\S^{d-2})\int_0^\delta (\sqrt{2} L - y\cos(\phi)) \sin(\phi)^{d-2} e^{\sqrt{2} y \cos(\phi)} \dd \phi \,,
\end{align*}
which is done using Laplace's method, as follows.
Let $\eta > 0$, and fix $C_d>0$ large enough and $\delta=\delta(\eta) > 0$ small enough such that for all $0 \leq \phi \leq \delta$, we have
\[
  1 - \phi^2/2 \leq \cos(\phi) \leq 1 - (1 - \eta)\phi^2/2 \text{ and } \phi^{d-2} - C_d \phi^{d} \leq \sin(\phi)^{d-2} \leq \phi^{d-2}\,.
\]
With this notation, we observe that
\begin{multline*}
  \left| I_{d,\delta}(L,y) - \mathrm{vol}(\S^{d-2})(\sqrt{2} L - y) \int_0^\delta \phi^{d-2} e^{\sqrt{2} y \cos(\phi)} \dd \phi \right|\\
  \leq (\sqrt{2}L+y)(2 C_d +1/2) \int_0^\delta \phi^d e^{\sqrt{2} y \cos(\phi)} \dd \phi\,.
\end{multline*}
From Laplace's method, we have that for all $k \geq 0$,
\[
  \int_0^\delta \phi^k e^{\sqrt{2}y \cos(\phi)} \dd \phi \sim_{y \to \infty} e^{\sqrt{2}y}  \frac{1}{2} \left( \frac{\sqrt{2}}{y}\right)^{(k+1)/2} \Gamma((k+1)/2)\,.
\]
As a result, uniformly over $\epsilon L \leq y\leq \sqrt{2} L - K \log L$, we have
\[
  (\sqrt{2} L - y) \int_0^\delta \phi^{d-2} e^{\sqrt{2} y \cos(\phi)} \dd \phi \gg (\sqrt{2}L+y)\int_0^\delta \phi^d e^{\sqrt{2}y \cos(\phi)} \dd \phi \,.
\]
Thus, for any fixed $K>0$ and uniformly in $y \in [\epsilon L, \sqrt{2} L - K \log L]$, we have
\begin{align*}
  I_{d,\delta}(L,y) &\sim_{L\to \infty} \mathrm{vol}(\S^{d-2})(\sqrt{2} L - y)  e^{\sqrt{2}y}  \frac{\Gamma((d-1)/2)}{2} \left( \frac{\sqrt{2}}{y}\right)^{(d-1)/2}\\
  &= (2\pi)^{(d-1)/4} y^{-(d+1)/2} (\sqrt{2}L-y) e^{\sqrt{2}y}\,.
\end{align*}
Similarly, if $y \in [\sqrt{2}L - K \log L, \sqrt{2} L + K \log L]$, we find that for some constant $C_{d} >0$,
\begin{equation}
  \label{eqn:upperbound}
  I_{d,\delta}(L,y) \leq C_d (\log L) L^{-(d-1)/2} e^{\sqrt{2}y}\,.
\end{equation}

Finally, using \eqref{eqn:farContrib} and \eqref{eqn:unifContinu}, we obtain that
\begin{align*}
 & \int_{\S^{d-1}} f(\theta) (\sqrt{2} L - y \psi \cdot \theta) e^{\sqrt{2}(y \psi \cdot \theta - \sqrt{2}L)} \sigma(\dd \theta) \sim_{L \to \infty} f(\psi) I_{d,\delta}(L,y) e^{-2 L}\\
  &\qquad\qquad\qquad \sim_{L \to \infty} f(\psi) (2\pi)^{(d-1)/4} y^{-(d+1)/2} (\sqrt{2}L-y) e^{\sqrt{2}(y - \sqrt{2}L)} \,,
\end{align*}
uniformly in $\psi \in \S^{d-1}$ and $y \in [\epsilon L, \sqrt{2} L - K \log L]$. The upper bound for $y \in [\sqrt{2} L - K \log L, \sqrt{2} L + K \log L]$ is obtained using \eqref{eqn:upperbound}
and \eqref{eqn:unifContinu}.
\end{proof}

Next, we show that particles that are not in the window $I_L^{\win}$ (defined in Section~\ref{subsec:trajectories}) at time $L$ do not contribute to the left-hand side of~\eqref{eqn:equiv-of-deriv-mtgs}.
\begin{lemma} \label{lem:non-window}
For any $x \in \R^d$,
\begin{align*}
  \lim_{L \to \infty} \sum_{u \in \mathcal{N}_L}
  \ind{R_L^{(u)} \not \in I_L^{\win}} (R_L^{(u)})^{-\alpha_d}  (1+|\sqrt{2}L - R_L^{(u)}|)  e^{-(\sqrt{2}L - R_L^{(u)}) \sqrt{2}}  = 0 \text{ in $\P_x$-probability.}
\end{align*}

\begin{proof}
For economy of notation, we prove this lemma for $x =0$. It will be abundantly clear that the same proof applies to an arbitrary $x$.

Let $Z_L'$ denote the expression on the left-hand side of the above display, and fix $\ep >0$. From a union bound, the many-to-one lemma (Lemma~\ref{lem:many-to-one}), the Bessel density at time $L$ \eqref{eqn:bessel-density}, and standard estimates of Gaussian integrals, it follows that there exists a constant  $K_d > 0$ depending only on the dimension $d$ such that
\begin{align}
    \P ( \exists u \in \cN_L : R_L^{(u)} > \sqrt{2}L + K_d \log L )& \leq
    \E \Big[ \sum_{u \in \cN_L} \ind{R_L^{(u)} > \sqrt{2}L + K_d \log L} \Big] \nonumber \\
    &= e^L \P( R_L > \sqrt{2}L +K_d \log L )
    \lesssim L^{-1/2}
    \label{eqn:parabola-ub} \,.
\end{align}
For brevity, let us write $B(L) := \sqrt{2}L+ K_d\log L$.  Then, using the Markov inequality and the many-to-one lemma (Lemma~\ref{lem:many-to-one}), we have
\begin{align*}
&\P( | Z_L' | > \ep ) \\
&\leq \      \ep^{-1}e^L \E \Big[ \ind{R_L \not \in I_L^{\win}, R_L \in [0,B(L)]}(R_L)^{-\alpha_d} (1+|\sqrt{2}L - R_L|) e^{-(\sqrt{2}L - R_L) \sqrt{2}}  \Big]
+o_L(1)\,,
\end{align*}
 where the  $o_L(1)$ term comes from~\eqref{eqn:parabola-ub}.
We now integrate over the density $p_L^R(0,\cdot)$ of $R_L$, so that the last display is given by
\begin{align}
    &\ep^{-1}e^L \int_{I} \frac{p_L^R(0,B(L)-w)}{(B(L)-w)^{\alpha_d}} (1+\abs{w- K_d \log L}) e^{-\sqrt{2}(w- K_d\log L)} \d w  +o_L(1) \nonumber \\
    &= c_d L^{-\frac{d}2} \ep^{-1} \int_{I} (B(L)-w)^{\alpha_d} (1+|w-K_d \log L|) e^{- \frac{(K_d \log L -w)^2}{2L}}\dd w\, + o_L(1)\,,
    \label{eqn:non-window-besseldensity}
\end{align}
where
\[
I:= [0,B(L)]\setminus [L^{1/6}+K_d \log L, L^{2/3}+K_d \log L]\,.
\]
The Gaussian term in the right-hand side of  \eqref{eqn:non-window-besseldensity} ensures that the latter is dominated by the integral over the interval $[0,L^{1/6}+K_d\log L]$. The integral over this interval (including the $c_dL^{-d/2} \ep^{-1}$ pre-factor) is bounded by a constant times $\ep^{-1} L^{-1/2} L^{1/6}L^{1/6} = o_L(1)$. Thus, we have shown that for any $\ep >0$, $\P(|Z_L'| > \ep)$ tends to $0$ as $L$ tends to $\infty$, which concludes the proof.
\end{proof}
\end{lemma}

\begin{proof}[Proof of Theorem~\ref{lem:equiv-of-deriv-mtgs}]
In this proof, the term ``with high probability'' means ``with
$\P_x$-probability approaching $1$ as $L\to\infty$''. Recall from~\eqref{eqn:sbm-convergence} that
\begin{align}
\sum_{u \in \calN_L} \int_{\sd} f(\theta) (\sqrt{2}L - \theta \cdot X_L^{(u)}) e^{\sqrt{2}\theta \cdot X_L^{(u)} - 2L} \sigma(\d\theta)
\xrightarrow{\P_x\text{-a.s.}} \langle f, D_{\infty} \rangle \,.
\label{eqn:equivalence-sbm-result}
\end{align}
We proceed by restricting the locations of the contributing particles at time $L$. To start, by \eqref{eqn:parabola-ub} with high probability, all norms of particles at time $L$ are bounded by $\sqrt{2}L + K_d \log L$, for some large, positive constant $K_d$. Furthermore, we can show that
\begin{align}
    \bigg| \E_x \bigg[ \sum_{u \in \calN_L} \ind{R_L^{(u)} \leq \sqrt{2}L-L^{2/3}} \int_{\S^{d-1}} f(\theta) (\sqrt{2}L - \theta \cdot X_L^{(u)}) e^{\sqrt{2}\theta \cdot X_L^{(u)} - 2L}
\sigma(\d\theta)\bigg] \bigg|
    \lesssim L^{1/6} e^{-\frac{L^{1/3}}{2}} \,.
    \label{eqn:equivalence-below-2/3}
\end{align}
Indeed, \eqref{eqn:equivalence-below-2/3} follows
by applying
the triangle inequality and the many-to-one lemma to its left-hand side, yielding an upper-bound of
\begin{align}
\label{eq-oldnonumber}
     e^L \norm{f}_{\infty} \int_{\S^{d-1}} \E \Big[ \big| \sqrt{2}L - (W_L+\theta\cdot x)\big | e^{- \sqrt{2}(\sqrt{2}L - (W_L+\theta\cdot x))} \ind{W_L+\theta\cdot x \leq \sqrt{2}L-L^{2/3}} \sigma(\d \theta) \Big]
    \,,
\end{align}
where $W_L$ denotes a standard one-dimensional Brownian motion at time $L$, and we have used the fact that $\theta \cdot (X_L^{(u)}-x)$ is equal to $W_L$ in distribution for any $\theta$ and $u$.
Taking $\tilde{W}_L := \sqrt{2}L - W_L$ and applying the Girsanov transform gives that \eqref{eq-oldnonumber} equals
\[
\|f\|_{\infty} \int_{\S^{d-1}} \E \Big[ |W_L-\theta\cdot x| \ind{W_L-\theta\cdot x \geq L^{2/3}}\Big] e^{\sqrt{2}\theta\cdot x}\sigma(\d \theta)\,,
\]
 from which~\eqref{eqn:equivalence-below-2/3} follows.

We can further restrict the locations of the $R_L^{(u)}$ by observing that
\begin{align*}
    \sum_{u \in \calN_L} \ind{R_L^{(u)} \in [\sqrt{2} L- K_d \log L, \sqrt{2} L + K_d \log L]} \int_{\sd} f(\theta) (\sqrt{2}L - \theta\cdot X_L^{(u)}) e^{\sqrt{2}\theta\cdot X_L^{(u)} - 2L}
\sigma(\d\theta)
\end{align*}
converges to $0$ in $\P_x$-probability due to the triangle inequality, the
upper bound~\eqref{eqn:laplacemeth-bad-x},
and Lemma~\ref{lem:non-window}. Together with~\eqref{eqn:equivalence-below-2/3} and~\eqref{eqn:parabola-ub}, we have thus far shown that
\begin{align*}
    &\langle D_L',f \rangle := \sum_{u \in \calN_L} \ind{R_L^{(u)} \in [\sqrt{2}L-L^{2/3}, \sqrt{2}L - K_d \log L] }
    \int_{\sd} f(\theta) (\sqrt{2}L - \theta\cdot X_L^{(u)}) e^{\sqrt{2}\theta\cdot X_L^{(u)} - 2L}\sigma(\d\theta)
\end{align*}
converges to $\langle D_\infty, f \rangle$ in probability as $L$ tends to infinity.
The uniform asymptotic in equation~\eqref{eqn:laplacemeth-final-good} of Lemma~\ref{lem:analysisTransform} shows that the above expression is equal to
\begin{align*}
    (1+o(1)) (2 \pi)^{(d-1)/4} \sum_{u \in \calN_L} f(\theta_L^{(u)}) \mathfrak{M}_L^{(u)}
    \ind{R_L^{(u)} \in [\sqrt{2}L-L^{2/3}, \sqrt{2}L - K_d \log L]}\,.
\end{align*}
Now, $\norm{f}_{\infty} <\infty$, and Lemma~\ref{lem:non-window} shows that
\begin{align*}
    \lim_{L\to\infty}\sum_{u \in \cN_L} \mathfrak{M}_L^{(u)}
    \ind{R_L^{(u)} \not \in I_L^{\win} } = 0 \text{ in $\P_x$-probability.} 
\end{align*}
This concludes the proof.
\end{proof}

\section{Proof of Theorem~\ref{thm:main}, Proposition~\ref{prop:Lalley and Sellke} and Proposition~\ref{prop:laplaceFunctional}}
\label{sec:proof-laplace-functional}

Proposition~\ref{prop:laplaceFunctional} will follow from Theorem~\ref{lem:equiv-of-deriv-mtgs} and the following estimate, which can be seen as a two-fold extension of~\cite[Theorem 3.2]{KLZ21}.

\begin{proposition}
\label{prop:tailEstimate}
Let $\phi: \S^{d-1}\times \R \to \R_{+}$ be a compactly supported, continuous function. For any $\theta \in \S^{d-1}$, define the (positive) constant
\begin{align}
    C_d(\phi_{\theta}) :=  \sqrt{\frac{2^{1+\alpha_d}}{\pi}}  C(\phi_{\theta})\,,
    \label{def:C_d}
\end{align}
where $\phi_{\theta}(y) := \phi(\theta,y)$, and  $C(\phi_{\theta})$ is defined in~\eqref{def:C(phi)}. Then
\begin{equation}
  \label{eqn:tailEstimate}
  \lim_{L \to \infty} \limsup_{t \to \infty} \sup_{\substack{z \in [L^{1/6},L^{2/3}],\\ \theta \in \sd }}\left|\frac{\E_{\theta (\sqrt{2}L-z)} \left[ 1 - \exp\left(-\crochet{\bar{\mathcal{E}}_{L,t},\phi} \right) \right]}{\mathfrak{M}_{L,z}} - C_d(\phi_{\theta})  \right| = 0,
\end{equation}
where
\[
  \bar{\mathcal{E}}_{L,t} := \sum_{u \in \mathcal{N}_{t-L}} \delta_{\big(\theta_{t-L}^{(u)},~R_{t-L}^{(u)} - m_t^{(d)} \big)} \quad \text{ and } \quad
  \mathfrak{M}_{L,z} := (\sqrt{2}L-z)^{-\alpha_d}z e^{-z\sqrt{2}}\,.
\]
\end{proposition}

Remark that when compared to Theorem~3.2 in~\cite{KLZ21},
Proposition~\ref{prop:tailEstimate} adds information
on
the direction of large particles.
It states that, if the initial particle is close to $\theta z$ for large $z \in \R_+$, then with high probability, the direction of the farthest particle from the origin will be in a small
neighbourhood of~$\theta$. Additionally, while the former result only deals with the tail of $R^*_{t-L}$, which can be seen as the Laplace functional computed with $\phi(x) = \infty \ind{x > 0}$, equation~\eqref{eqn:tailEstimate} extends  the Laplace functional to a broader class of functions.

In what follows, we show how Theorem~\ref{lem:equiv-of-deriv-mtgs} and Proposition~\ref{prop:tailEstimate} imply Proposition~\ref{prop:laplaceFunctional}.
We then use Proposition~\ref{prop:laplaceFunctional} to prove Theorem~\ref{thm:main}. The proof of Proposition~\ref{prop:tailEstimate} is postponed to Section~\ref{sec:proof-tail-estimate}.

\subsection{Proof of Proposition~\ref{prop:laplaceFunctional} using Proposition~\ref{prop:tailEstimate}}

\begin{proof}[Proof of Proposition~\ref{prop:laplaceFunctional}]
Fix $x \in \R^d$ and some $\phi: \S^{d-1} \times \R \to \R_{\geq 0}$ satisfying the conditions of Proposition~\ref{prop:laplaceFunctional}.
Let $y$ be the minimum of the support of $\phi$ on $\R$.
We are interested in the quantity
\[
  g_t(\phi) := \E_x \bigg[ \exp \bigg( -\sum_{u \in \cN_t} \phi(\theta_t^{(u)}, R_t^{(u)} - m_t^{(d)}) \bigg) \bigg] \,.
\]
Remark that the branching Brownian motion under law $\P_x$ can be constructed as $(x + Y_t^{(u)}, u \in \mathcal{N}_t)_{t \geq 0}$ where $Y$ is a branching Brownian motion started from $0 \in \R^d$. In particular, we have
\[
  R^{(u)}_t = \|Y_t^{(u)} + x\| \in \big[\|Y_t^{(u)}\| - \|x\|, \|Y_t^{(u)}\| + \|x\|\big] \,.
\]
Using the above display and the fact that $\phi(\theta_t^{(u)}, R_{t}^{(u)} - m_t^{(d)})$ is non-zero only if $R_t^{(u)} > m_t^{(d)}+y$, we obtain by Proposition~\ref{prop:window-time-L} that
only particles $u \in \cN_t$ that are descendants of particles in $\cN_L^{\win}$  contribute. Thus, applying the branching property at time $L$ as well as the Markov property shows that
\[
  g_t(\phi) = \E_x\bigg[ \prod_{u \in \calN_L^{\win}} \E_{X_L^{(u)}} \left( \exp\left(-\crochet{ \bar{\mathcal{E}}_{L,t},\phi} \right) \right) \bigg] +o_u(1) \,.
\]
Set $z_L(u) := \sqrt{2}L- R_L^{(u)}$. Then,
using Proposition~\ref{prop:tailEstimate}, we can follow the steps leading up to~\cite[Equation~3.7]{KLZ21} to obtain a non-negative sequence~$\ep_L$ tending to $0$ as $L$ tends to infinity such that for all $L$ large and $t$ large enough compared to $L$, we have
\begin{align*}
    \Bigg|g_t(\phi) -  \E_x\bigg[ \exp\Big( - \!\!\!\!\!  \sum_{u \in \calN_L^{\win}} C_d(\phi_{\theta_L^{(u)}}) \mathfrak{M}_{L,z_L(u)}
    \Big) \bigg] \Bigg| \leq \ep_L \,.
\end{align*}
From Theorem~\ref{lem:equiv-of-deriv-mtgs}, we have convergence of the above Laplace transform to
\[
\E_x\bigg[ \exp\Big( - (2\pi)^{-\alpha_d/2}\int_{\sd} C_d(\phi_{\theta}) D_{\infty}(\theta) \sigma(\d\theta) \Big) \bigg]\,,
\]
as $L$ tends to infinity.
Recalling now the definition of $C_d(\phi_{\theta})$ from~\eqref{def:C_d} and of $C(\phi)$ from~\eqref{def:C(phi)},
the previous two displays yield a non-negative sequence $\ep_L'$ tending to $0$ as $L$ tends to infinity such that for all $L$ large and $t$ large enough compared to $L$, we have
\begin{align*}
    \Bigg|g_t(\phi) -   \E_x\bigg[ \exp\Big( - \sqrt{\frac{2}{\pi^{1+\alpha_d}}} \int_{\sd} C(\phi_{\theta}) D_{\infty}(\theta)\sigma(\d\theta) \Big) \bigg] \Bigg| \leq \ep_L' \,.
\end{align*}
Since $g_{t}(\phi)$ has no $L$-dependence, it follows from the previous two displays that
\begin{align}
    g_t(\phi) \xrightarrow[t \to \infty]{}  \E_x\Big[ \exp\Big( - \sqrt{\frac{2}{\pi^{1+\alpha_d}}}\int_{\sd} C(\phi_{\theta}) D_{\infty}(\theta) \sigma(\d\theta) \Big) \Big]\,,
\end{align}
as desired.
\end{proof}

\subsection{Proof of Theorem~\ref{thm:main} and Proposition~\ref{prop:Lalley and Sellke}}
\label{sec:proof}
We recall that, as observed in Remark~\ref{rem:EVERYONE}, the statement of Proposition~\ref{prop:laplaceFunctional} is valid in all dimensions,
including $d=1$,
where for $d=1$, we have $\S^{d-1} = \{-1,1\}$, and $\sigma = \delta_1+ \delta_{-1}$ (in other words, in all dimensions $\sigma$ is the Haar measure on $\S^{d-1}$).

\begin{proof}[Proof of Theorem~\ref{thm:main}]
We note that Proposition~\ref{prop:laplaceFunctional} already gives the claimed convergence, so that we only need to prove the claimed identification of the limit.

By~\cite{ABBS12} or~\cite{ABK13}, in dimension $1$, for all continuous compactly supported function $\varphi$ on $\mathbb{R}$, we have that
\[
\lim_{t \to \infty}\E\left( e^{-\sum_{u \in \mathcal{N}_t} \varphi(X^{(u)}_t - m_t^{(1)})} \right) = \E\left( e^{- \crochet{\mathcal{E}^{(1)}_\infty,\varphi}} \right),
\]
with $\mathcal{E}^{(1)}_\infty$ a SDPPP($\sqrt{2} \gamma D_\infty$,$e^{-\sqrt{2}x}$, $\mathcal{D}$), where $m_t^{(1)}$ denotes $m_t^{(d)}$ corresponding to dimension $d=1$ and $D_\infty$ is the same as $D_\infty(\theta)$ for $d=1$ with
$\theta=1$, i.e.,  it is  the standard one dimensional derivative martingale.
Using Campbell's formula on the right hand side of the last display, we
obtain that
\[
  \E\left( e^{- \crochet{\mathcal{E}^{(1)}_\infty,\varphi}} \right) = \E\left( \exp\left( -  \sqrt{2} \gamma
  D_\infty\int_\R \left(1 - \E\left( e^{-\crochet{\mathcal{D},\varphi(x+\cdot)}} \right) \right) e^{-\sqrt{2}x}\dd x  \right) \right).
\]
Then using Proposition~\ref{prop:laplaceFunctional}
in dimension $d=1$ with $\varphi(x)\ind{\theta=1}$,
we obtain that
\begin{equation}
  \label{eq-ABKKLZ}
  \lim_{t \to \infty} \E\left( e^{-\crochet{\mathcal{E}^1_t,\varphi}} \right) = \E\Big[ \exp\Big( - \fC_1 C(\varphi)  D_{\infty} \Big) \Big],
\end{equation}
where $\fC_1 = \sqrt{2/\pi}$. (Alternatively, \eqref{eq-ABKKLZ} is obtained in~\cite[Proposition 3.2]{ABK13}.) As a result, we have
\[
   \E\Big[ \exp\Big( - \fC_1 C(\varphi) D_{\infty} \Big) \Big]
  =
  \E\left( \exp\left( -  \sqrt{2} \gamma D_\infty\int_\R \left(1 - \E\left( e^{-\crochet{\mathcal{D},\varphi(x+\cdot)}} \right) \right) e^{-\sqrt{2}x}\dd x  \right) \right)\,,
\]
and therefore
\begin{equation}
  \label{eqn:identificationOfC}
  \fC_1 C(\varphi) =  \sqrt{2} \gamma\int_\R \left(1 - \E\left( e^{-\crochet{\mathcal{D},\varphi(x+\cdot)}} \right) \right) e^{-\sqrt{2}x}\dd x,
\end{equation}
using that the Laplace transform of $ D_\infty$
is strictly decreasing as $ D_\infty$ is non-negative and non-degenerate.

Returning to $d>1$,
thanks to this identification, we can then compute, for $\phi$ continuous
and compactly supported on $\S^{d-1}\times \R$, again by Campbell's formula,
\begin{align*}
  &\E\left[ e^{-\crochet{\mathcal{E}_\infty,\phi}} \right]\\
  &=  \E\Big[ \exp\Big( - \sqrt{2}\gamma \pi^{-\alpha_d/2} \int_{\S^{d-1}} D_{\infty}(\theta) \int_\R \left(1 - \E\left[ e^{-\crochet{\mathcal{D},\phi_\theta(x+\cdot)}} \right] \right) e^{-\sqrt{2}x}\dd x \sigma(\d\theta) \Big) \Big]\\
  &=   \E\Big[ \exp\Big( - \fC_d \int_{\S^{d-1}} C(\phi_{\theta}) D_{\infty}(\theta) \sigma(\d\theta) \Big) \Big],
\end{align*}
by \eqref{eqn:identificationOfC}, using that $\fC_d = \pi^{-\alpha_d/2} \fC_1$.
\end{proof}

We remark that Theorem~\ref{thm:main} can be straightforwardly extended to a branching Brownian motion started from a finite number of particles at positions $x_1,\ldots,x_n$ in $\R^d$ using the branching property of the branching Brownian motion and the superposition property of Poisson point processes. More precisely, for $n \in \N$ and $\bfx = (x_1,\ldots,x_n) \in (\R^{d})^n$, we denote by $\P_\bfx$ the law of a branching Brownian motion started from $n$ initial particles at position $x_1$, ..., $x_n$. Under law $\P_\bfx$, $D_\infty$ has the law of $\sum_{j=1}^n D^{(j)}_\infty$, where $(D^{(j)}_\infty,~j\leq n)$ is a collection of independent random variables such that each $D^{(j)}_\infty$ has the same law as $D_\infty$ under $\P_{x_j}$. The following result then holds.
\begin{corollary}
\label{cor:multiplePoints}
Under law $\P_\bfx$, the extremal point process $\mathcal{E}_t$ converges in distribution for the topology of vague convergence to the decorated Poisson point process $\mathcal{E}_\infty$, which is the process described in Theorem~\ref{thm:main} (up to $D_\infty$ being distributed under law $\P_\bfx$).
\end{corollary}
\begin{proof}
We observe that we can write
$
  \mathcal{E}_t = \sum_{j=1}^n \mathcal{E}^{(j)}_t,
$
where $\mathcal{E}^{(j)}_t$ is the extremal point process consisting of the descendants of the initial particle alive at position $x_j$. By the branching property, $(\mathcal{E}_t^{(j)},~j \leq n)$ are independent. Using Theorem~\ref{thm:main}, we conclude the convergence of the   joint distribution of the processes $\mathcal{E}^{(j)}_t$ to that of the $\mathcal{E}^{(j)}_\infty$, where $\mathcal{E}^{(j)}_\infty$ is the process described in Theorem~\ref{thm:main}, with $D_{\infty}$ replaced by $D_{\infty}^{(j)}$. Using the superposition property of Poisson point processes, the proof is now complete.
\end{proof}

As an application of Theorem~\ref{thm:main}, we obtain the convergence in distribution of the centered norm of the maximal displacement at a large time $t$. This generalizes \cite[Theorem~1.1]{KLZ21}, as the convergence holds for any initial condition of the BBM.
%this process.
% position $x \in \R^d$ for the process.
\begin{corollary}
\label{cor:cvd}
We write $(\bar{\theta}_\infty,\bar{\rho}_\infty)$ for the atom of $\mathcal{E}_\infty$ with largest second coordinate. For all $\bfx \in (\R^d)^n$, we have
\[
  \lim_{t \to \infty} (\theta^*_t,R^*_t-m_t^{(d)}) = (\bar{\theta}_\infty, \bar{\rho}_\infty) \text{ in $\P_\bfx$-distribution.}
\]
In particular, for all $x \in \R^d$ and $y \in \R$, we have
\[
  \lim_{t \to \infty} \P_x\left( R^*_t \leq m_t^{(d)} + y\right) = \E_x\left( \exp\left( -\gamma \pi^{-\alpha_d/2} D_\infty(\S^{d-1})e^{-\sqrt{2}y}\right)\right).
\]
\end{corollary}

\begin{proof}
We use the same technique as in the proof of \cite[Lemma 4.4]{BBCM22}, where the vague convergence in distribution of $\mathcal{E}_t$ and the tightness of its largest atom is sufficient to conclude that the convergence in distribution of the largest atom holds.

Let $\bfx \in (\R^d)^n$. By \cite[Theorem 1.1]{Mallein15}, for all $\epsilon > 0$, there exists $M> 0$ such that
\[
  \P_\bfx(|R^*_t - m_t^{(d)}| \geq M) \leq \epsilon \,,
\]
using a trivial union bound. Then, writing $(\bar{\theta}^{(M)}_t,\bar{\rho}^{(M)}_t)$ for the atom of  $\mathcal{E}_t$ having the largest second component among all atoms belonging to $[-M,M]$ (setting $(\bar{\theta}^{(M)}_t,\bar{\rho}^{(M)}_t) = (0,-1)$ if this set is empty), we have
\[
  \P_\bfx((\bar{\theta}^{(M)}_t,\bar{\rho}^{(M)}_t) \in A) - \epsilon \leq \P_\bfx((\theta^*_t, R^{*}_t-m^{(d)}_t) \in A) \leq \P_\bfx((\bar{\theta}^{(M)}_t,\bar{\rho}^{(M)}_t) \in A) + \epsilon
\]
for every measurable subset $A$ of $\S^{d-1}\times \R_+$.

By Theorem~\ref{thm:main}, letting $t \to \infty$ we have $(\bar{\theta}^{(M)}_t,\bar{\rho}^{(M)}_t) \to (\bar{\theta}^{(M)}_\infty,\bar{\rho}^{(M)}_\infty)$ in distribution, which converges to $(\bar{\theta}_\infty,\bar{\rho}_\infty)$ as $M \to \infty$. Finally, letting $\epsilon \to 0$ and using the Portemanteau theorem, we deduce that $(\theta^*_t, R^*_t - m_t^{(d)})$ converges in distribution to $(\bar{\theta}_\infty,\bar{\rho}_\infty)$. We complete the proof using the definition of $\mathcal{E}_\infty$.
\end{proof}
%
%\corRef{This corollary is a consequence of the fact that the law of $(\theta^*_t, R^*_t - m_t^{(d)})$ converges in distribution to the law of the atom of $\mathcal{E}_\infty$ with largest second coordinate. However, as we only prove the convergence of $\mathcal{E}_t$ to $\mathcal{E}_\infty$ for the topology of the \propBM{vague} %weak
%convergence, the desired convergence in distribution does not immediately stem from Theorem~\ref{thm:main}. Rather, we use the same technique as in the proof of \cite[Lemma 4.4]{BBCM22}, where the convergence in distribution of $\mathcal{E}_t$ and the tightness of $(\theta^*_t, R^*_t - m_t^{(d)})$ is sufficient to conclude that this convergence in distribution holds.}

%\corRef{Let $x \in \R^d$. By \cite[Theorem 1.1]{Mallein15}, for all $\epsilon > 0$, there exists $M>0$ such that
%\[
%  \P_x(R^*_t - m_t^{(d)} \geq M) \leq \epsilon \,.
%\]
%Then, for all $y \in \R$, we have
%\[
%  \P_x \big(\mathcal{E}_t(\S^{d-1} \times [y,M]) = 1\big) - \epsilon \leq \P_x(R^*_t - m_t^{(d)} \leq y) \leq \P_x \big(\mathcal{E}_t(\S^{d-1} \times [y,M]) = 0\big) + \epsilon\,.
%\]
%By Theorem~\ref{thm:main}, letting $t \to \infty$ then $M \to \infty$ and finally $\epsilon \to 0$, we obtain
%\begin{equation*}
%  \lim_{t \to \infty} \P_x(R^*_t - m_t^{(d)} \leq y)
%  = \P_x\big(\mathcal{E}_\infty(\S^{d-1} \times [y,\infty)) = 0\big)\,,
%\end{equation*}
%which completes the proof by definition of $\mathcal{E}_\infty$.}
%\end{proof}

Finally, using Corollary \ref{cor:cvd}, we prove Proposition~\ref{prop:Lalley and Sellke} using the Markov property and the convergence of bounded martingales.
\begin{proof}[Proof of Proposition~\ref{prop:Lalley and Sellke}]
Let $L \geq 0$, using the branching property at time $L$, we have
\begin{align*}
  \P(R^*_t \leq m_t^{(d)} + y|\mathcal{F}_L) &= \prod_{u \in \mathcal{N}_L} \P_{X_L^{(u)}}\left( R^*_{t-L} \leq m_t^{(d)} + y\right)\\
  &\xrightarrow{t \to \infty} \prod_{u \in \mathcal{N}_L} \E_{X_L^{(u)}}\left( \exp\left(-\gamma \pi^{-\alpha_d/2} D_\infty(\S^{d-1})e^{-\sqrt{2}y - 2L} \right)\right) \text{ a.s.}
\end{align*}
using that $\lim_{t\to \infty} m_{t}^{(d)} - m_{t-L}^{(d)} = \sqrt{2}L$ and Corollary \ref{cor:cvd}.

On the other hand, for all $L \leq t$, we have
\[
  (D_t(\theta),\theta \in \S^{d-1}) = \left(e^{-2L}\sum_{u \in \mathcal{N}_L} D^{(u)}_{t-L}(\theta) + \sqrt{2} L e^{-2L}\sum_{u \in \mathcal{N}_L} W^{(u)}_{t-L}(\theta), \theta \in \S^{d-1}\right),
\]
with the following notation. For $u \in \mathcal{N}_L$, recall the set $\mathcal{N}^{u}_{t-L}$ (see just above Proposition~\ref{prop:sector-result}), and for $v \in \mathcal{N}^{u}_{t-L}$, let $X^{(v)}_{t}$ denote the position of that particle at time $t$. Then, for $u \in \mathcal{N}_L$, the $(D^{(u)}_{t-L},W^{(u)}_{t-L})$ are defined via
\[
  \theta\mapsto \left(\sum_{v \in \mathcal{N}_{t-L}^{u}} (\sqrt{2} (t-L) - X_{t}^{(v)}\cdot\theta)e^{\sqrt{2}X^{(v)}_{t} \cdot \theta - 2(t-L)}, \sum_{v \in \mathcal{N}_{t-L}^{u}} e^{\sqrt{2}X^{(v)}_{t} \cdot \theta - 2(t-L)} \right).
\]
By the branching property, $(D^{(u)}_{t-L},W^{(u)}_{t-L})_{u \in \mathcal{N}_L}$ are conditionally on $\mathcal{F}_L$ independent random variables, with each $(D_{t-L}^{(u)}, W_{t-L}^{(u)})$ having the same law as $(D_{t-L},W_{t-L})$ under law $\P_{X_L^{(u)}}$. Therefore, for all $u\in\mathcal{N}_L$, the $(D_t^{(u)},W_t^{(u)})$ jointly converge almost surely as $t \to \infty$ to $(D^{(u)}_\infty,0)$ for the topology of the weak convergence on the sphere, by \cite[Theorem 1.3]{StBeMa2020}. Moreover, conditionally on 
%$\mathcal{F}_\infty$ \corRef{YK: 
$\mathcal{F}_L$, $(D^{(u)}_\infty, u \in \mathcal{N}_L)$ are independent random variables such that $D^{(u)}_\infty$ has the same law as $D_\infty$ under the
law $\P_{X_L^{(u)}}$.
As a result, letting $t \to \infty$, we deduce that for all $L> 0$, the following measures coincide a.s.:
\begin{equation}
  \label{eqn:smoothingTransform}
   D_\infty(\theta)\sigma(\d \theta) = e^{-2L}\sum_{u \in \mathcal{N}_L}D^{(u)}_{\infty}(\theta)\sigma(\d \theta) \,.
\end{equation}

In particular, this yields
\begin{multline*}
  \prod_{u \in \mathcal{N}_L} \E_{X_L^{(u)}}\left( \exp\left(-\gamma \pi^{-\alpha_d/2} D_\infty(\S^{d-1})e^{-\sqrt{2}y - 2L}\right) \right)\\
  = \E\left(\exp\left(-\gamma \pi^{-\alpha_d/2} D_\infty(\S^{d-1})e^{-\sqrt{2}y - 2L}\right) \middle| \mathcal{F}_L\right) \quad \text{a.s.}
\end{multline*}
Using the almost sure convergence of closed martingales, we conclude that
\[
  \lim_{L \to \infty} \lim_{t \to \infty} \P(R^*_t \leq m_t^{(d)} + y|\mathcal{F}_L) = \exp\left(-\gamma \pi^{-\alpha_d/2} D_\infty(\S^{d-1})e^{-\sqrt{2}y - 2L}\right)
\]
almost surely.

We now turn to the convergence of the angle of the maximal displacement. Using again the branching property at time $L$ then Corollary \ref{cor:cvd}, we have for any measurable set $A \subset \S^{d-1}$ such that $\sigma(\partial A)=0$,
\begin{align*}
  \P(\theta^*_t \in A|\mathcal{F}_L) &= \P_{(X_L^{(u)},u \in \mathcal{N}_L)}(\theta^*_t \in A)\\
  &\xrightarrow{t \to \infty} \P_{(X_L^{(u)},u \in \mathcal{N}_L)}(\bar{\theta}_\infty \in A) = \E_{(X_L^{(u)},u \in \mathcal{N}_L)}\left(\frac{D_\infty(A)}{D_\infty(\S^{d-1})}\right) \text{ $\P$-a.s.,}
\end{align*}
by the definition of $\mathcal{E}_\infty$. Hence using \eqref{eqn:smoothingTransform}, we deduce that
\[
  \E_{(X_L^{(u)},u \in \mathcal{N}_L)}\left(\frac{D_\infty(A)}{D_\infty(\S^{d-1})}\right) = \E\left( \frac{D_\infty(A)}{D_\infty(\S^{d-1})} \middle|\mathcal{F}_L\right) \quad \text{a.s.}.
\]
We can now conclude using again the convergence of closed martingales.
\end{proof}

\section{Proof of Proposition~\ref{prop:tailEstimate}}
\label{sec:proof-tail-estimate}

We finish this article with a proof of Proposition~\ref{prop:tailEstimate}, which uses the geometrical result Proposition~\ref{prop:sector-result} to take care of the directional constraint, the results of~\cite{KLZ21} to characterize the typical trajectories of particles contributing to the Laplace functional, and a coupling with one-dimensional branching Brownian motion on the last time interval of length $\ell$. Throughout this section, let $\ell:=\ell(L)$ satisfy the properties
\begin{align}
1\leq \ell \leq L^{1/6} \quad \text{ and } \quad  \lim_{L \to \infty} \ell= \infty \,.
\label{def:ell}
\end{align}
The precise dependence of $\ell$ on $L$ does not play a role. For convenience, we will write $\tl := t - L$. Since we will send $t$ to infinity before $L$, the parameter $L$ should be thought of as an order one quantity (compared to $t$). Also, we fix a constant $y \in \R$ in what follows until further specified.

\subsection{Description of the extremal particles by \texorpdfstring{\cite{KLZ21}}{KLZ21}}
The (modified) second moment method used to prove~\cite[Theorem~3.2]{KLZ21} implies the following:
for a branching Bessel process started from the window $I_L^{\win}$ (at time $0$),
particles $u \in \cN_{\tl}$ that reach height $m_t^{(d)}+y$ or higher at time $\tl$ follow a well-controlled trajectory until time $\tll$; furthermore, among particles $v\in \cN_{\tll}$ that follow such a trajectory, at most one produces a descendent that reaches height $m_t^{(d)}+y$. These two results are reproduced below as Propositions~\ref{prop:klz21-1stmom-results} and~\ref{prop:klz21-2ndmom}, respectively.
We will use these two ideas in the proof of Proposition~\ref{prop:tailEstimate} to great effect.
Before we begin, let us lay out some notation that will be familiar from~\cite{KLZ21}.

Let
$\y(b) := \frac{m_t^{(d)}}{t}(t-\ell) +y - b$. For any $v \in \cN_{\tll}$, we define the event
\begin{align}
    \mathfrak{T}(v) &:= \mathfrak{T}_{t, L, y}(v) =   \Big \{ \max_{u \in \cN_{\ell}^v} R_{\tl}^{(u)} > m_t^{(d)}+y
    \Big \} \, .
    \label{def:T-tailevent}
\end{align}
We will also make use of certain ``barrier events'' to restrict the paths of our BBM particles. For functions $f,g: [0,\infty) \to \R$, a set  $I \subset [0,\infty)$,
and a real-valued process $X_\cdot$,
we call events of the following form \emph{barrier events}:
\begin{align*}
    \UB{I}{f}(X_\cdot) &:= \{ X_u \leq f(u),~\forall u \in I \} \quad \text{ and }
    \quad
    \LB{I}{f} (X_\cdot) := \{ X_u \geq f(u),~\forall u \in I \}
    \, .
\end{align*}
Recall the barriers $B_0(s):=B_0(s; t,L)$ (\cite[Equation~4.38]{KLZ21}) and $Q_z(s):= Q_z(s; t, L,y)$ (\cite[Equation~5.5]{KLZ21}), whose precise definitions we will not use here.
A crucial barrier event used throughout this subsection will be the event
that a process $X_{\cdot}$ is bounded above by the linear barrier $\frac{m_t^{(d)}}{t}(\cdot+L)+y$ (where we write $f(\cdot+r)$ to denote the function $u \mapsto f(u+r)$) and bounded below by $Q_z(\cdot)$ on a certain time interval $I \subset \R_{\geq 0}$. We will denote this event by $\B_I(X_\cdot) := \B_{I,y,z,L,t}(X_\cdot)$; note that
\begin{align}
    \B_I(X_\cdot) =
    \UB{I}{\frac{m_t^{(d)}}{t}(\cdot+L)+y}(X_\cdot)
    \cap
    \LB{I}{Q_z}(X_\cdot)
    \, .
    \label{eqn:def-upper/lowerbarrier}
\end{align}
Now, for any $v \in \cN_{\tll}$, define the events
\begin{align}
    F_{L,t}(v)
    &:=
        \UB{[0, \tl - \ell]}{B_0}(R^{(v)}_\cdot)
        \cap
        \Big \{R^{(v)}_{\tll} > \frac{t}{\sqrt{d}} \Big \}
    \, , \text{ and} \label{def:F-event}\\
    G_{L,t}(v)
    &:=
    \B_{[0, \tl-\ell]}(R^{(v)}_\cdot)
        \cap
        \left\{
            \y\big( R^{(v)}_{\tl - \ell} \big)  \in [ \ell^{1/3}, \ell^{2/3}]
        \right\}
        \,.  \label{def:G-event}
\end{align}
The event $G_{L,t}(v) \cap \fT(v)$ is depicted in Figure~\ref{fig:G-Lt-event}.
\begin{figure}
     \begin{tikzpicture}[>=latex,font=\small]
 \draw[->] (0, 0) -- (10, 0);
  \draw[->] (0, 0) -- (0, 6.5);
  \node (fig1) at (4.55,3.8) {
    \includegraphics[width=.65\textwidth]{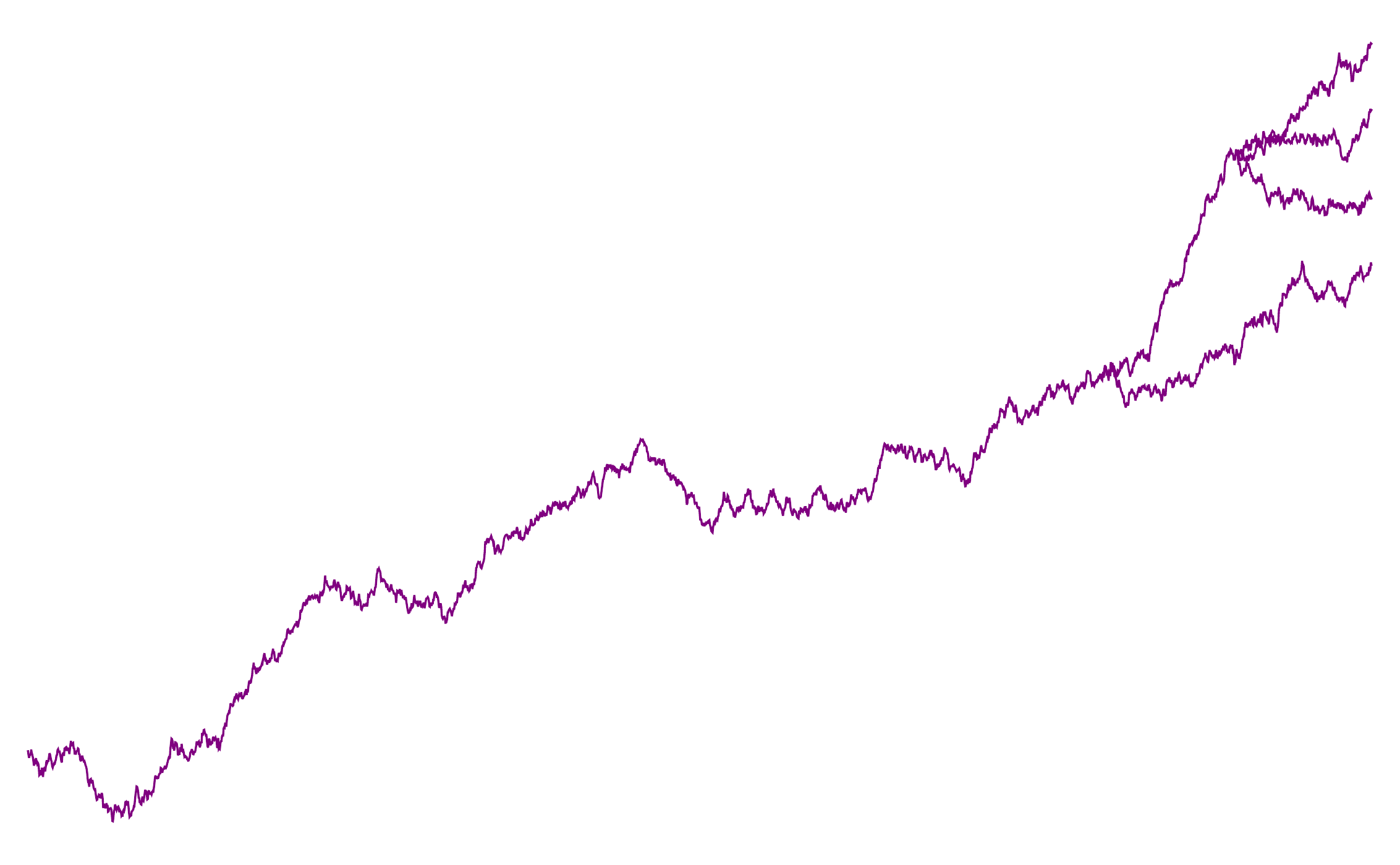}};
   \node[circle,fill=blue!75!black,label={[blue!75!black]left:$\frac{m_t^{(d)}}tL+y$}, inner sep=1.75pt] (a) at (0.0,2.5) {};
   \node[circle,fill=blue!75!black,label={[blue!75!black]right:$m_t^{(d)}+y$},inner sep=1.75pt] (b) at (9.12,5.6) {};
   \draw (-0.05,2.1) -- (0.05,2.1) node [left,font=\tiny] {$\sqrt 2 L - L^{1/6}$};
   \draw (-0.05,1.3) -- (0.05,1.3) node [left,font=\tiny] {$\sqrt 2 L - L^{2/3}$};
   \draw (-0.05,.6) -- (0.05,.6) node [left,font=\tiny] {$\sqrt 2 L - 2L^{2/3}$};
   \node[circle,fill=blue!45!purple,label={[blue!45!purple]left:$\sqrt{2}L-z$},inner sep=1.75pt] (x) at (0, 1.7) {};
   \node[circle, fill=blue!75!black, label={[blue!75!black, font=\tiny,label distance=-3pt]above left:$\y(0)$}, inner sep=1.75pt] (y0) at ( $(a)!0.74!(b)$) {};
   \node[circle, fill=black, label={[black, font=\tiny,label distance=-3pt]left:$\y(\ell^{1/3})$}, inner sep=1.75pt] (y1) at ( $(y0)-(0,0.5)$) {};
   \node[circle, fill=black, label={[black,font=\tiny]right:$\y(\ell^{2/3})$}, inner sep=1.75pt] (y2) at ( $(y0)-(0,1.5)$) {};
   \node[circle, fill=green!50!blue, label={[green!50!blue,font=\tiny]right:$\y(2\ell^{2/3})$}, inner sep=1.75pt] (y3) at ( $(y0)-(0,3)$) {};
   \node[circle, fill=green!50!blue, inner sep=1.75pt] (a0) at ( 0.75, 0.6 ) {};
   \node[circle, fill=green!50!blue, inner sep=1.75pt] (a00) at ( 0, 0.6 ) {};
   \node[circle, color=green!50!blue,fill=green!10, draw, inner sep=1.75pt] (a1) at ( 0.75, 0.3 ) {};
   \node[circle, color=green!50!blue,fill=green!10, draw, inner sep=1.75pt] (b1) at ( 5.5, 1.1 ) {};
   \node[circle, fill=green!50!blue, inner sep=1.75pt] (b2) at ( $(y3)-(1.25,0)$ ) {};
  \draw[green!50!blue, thick] (a00) -- (a0);
  \draw[green!50!blue, thick, dashed] (a00) to[bend right=10] (a1);
  \draw[green!50!blue, thick] (a1) to[bend right=10] node[midway,above,sloped] {$Q_z(s)$} (b1);
  \draw[green!50!blue, thick, dashed] (b1) to[bend right=10] (y3);
  \draw[green!50!blue, thick] (b2) -- (y3);
  \draw[blue!75!black, thick] ($(a)+(0.02,0.02)$) -- node[midway,above,sloped] {$\frac{m_t^{(d)}}t(s+L)+y$} (y0);
  \draw[blue!75!black, thick, dashed] (y0) -- (b);
  \draw (.75,0.05) -- (.75,-0.05) node[below] {$\ell_1$};
  \draw (5.5,0.05) -- (5.5,-0.05) node[below] {$\tilde t-\ell-\ell_1$};
  \draw (6.75,0.05) -- (6.75,-0.05) node[below] {$\tilde t - \ell$};
  \draw (9.12,0.05) -- (9.12,-0.05) node[below] {$\tilde t$};
  \end{tikzpicture}
     \caption{The event $G_{L,t}(v)\cap \fT(v)$ from~\eqref{def:G-event}: the trajectory $R^{(v)}_.$ of a given particle $v\in \cN_{\tll}$ is bounded above by the blue line and below by the solid green curves that comprise $Q_z$ on $[0,\tll]$, $R^{(v)}_{\tll}$ lies in the window $[\y(\ell^{2/3}), \y(\ell^{1/3})]$, and $v$ produces a descendent in $\cN_{\tl}$ that exceeds $m_t^{(d)}+y$.}
     \label{fig:G-Lt-event}
\end{figure}
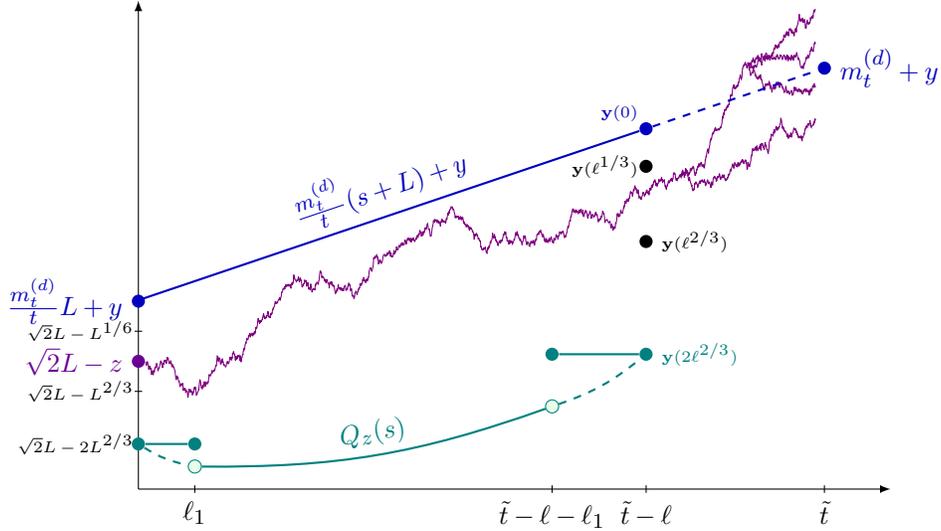
In Proposition~\ref{prop:tailEstimate}, we consider a $d$-dimensional BBM started from a point on the sphere of radius $\sqrt{2}L - z$, where $z \in [L^{1/6}, L^{2/3}]$. In particular, we seek to demonstrate asymptotic statements uniformly over $z$ in this interval. For such statements, we will use heavily the notation set forth in Section~\ref{subsec:asymp-notation}.

The following results  are from~\cite{KLZ21}.
\begin{proposition}[{\cite{KLZ21}}] \label{prop:klz21-1stmom-results}
\begin{align*}
    \E_{\sqrt{2}L-z} \bigg[ \sum_{u \in \cN_{\tll}} \indset{F_{L,t}(v)^c} \bigg]  &= o_u(\fM_{L,z}) \,, \text{ and} \\
    \E_{\sqrt{2}L-z} \bigg[ \sum_{u\in \cN_{\tll}} \indset{F_{L,t}(v) \setminus G_{L,t}(v)} \indset{\fT(v)} \bigg] &= o_u(\fM_{L,z})\,.
\end{align*}
\end{proposition}
Proposition \ref{prop:klz21-1stmom-results} follows from~\cite[Lemma~4.3, Lemma~5.1 and Claim~5.5]{KLZ21} respectively and their proofs.\footnote{The upper-bounds in these results of~\cite{KLZ21} are stated for quantities of the form $\P( \cup_{u\in \cN_{\tll}} A_v)$, while Proposition~\ref{prop:klz21-1stmom-results} bounds (larger) first-moment quantities of the form $\E[\sum_{u \in \cN_{\tll}} \indset{A_v}]$. However, the first step in the proof of each result of~\cite{KLZ21} is to bound $\P( \cup_{u\in \cN_{\tll}} A_v)$ by the corresponding first-moment quantity, so that the results there are actually shown via upper-bounds on the $\E[\sum_{u \in \cN_{\tll}} \indset{A_v}]$. Thus, we indeed have Proposition~\ref{prop:klz21-1stmom-results}.}
The result below is from~\cite[Lemma~5.2]{KLZ21} (see also there equation (7.2)).
\begin{proposition}[\cite{KLZ21}] \label{prop:klz21-2ndmom}
\begin{align*}
\E_{\theta (\sqrt{2}L -z )}\Big[ \sum_{v \neq w \in \cN_{\tll}} \ind{G_{L,t}(v) \cap G_{L,t}(w)} \Big] = o_u(\fM_{L,z})
\end{align*}
\end{proposition}

In the proof of Proposition~\ref{prop:tailEstimate}, we will use Proposition~\ref{prop:klz21-1stmom-results} to show that the only particles in $\cN_{\tl}$ contributing to $\E_{\theta(\sqrt{2}L-z)}[ 1- \exp (- \langle \overline{\cE}_{L,t} , \phi \rangle ) ]$ are those that descended from particles $v\in \cN_{\tll}$ that performed the event $G_{L,t}(v)$. Proposition~\ref{prop:klz21-2ndmom} shows that only one such particle in $v\in \cN_{\tll}$ will perform $G_{L,t}(v)$.

\subsection{Coupling with a one-dimensional BBM}
Eventually, we will encounter an expression of the form
\begin{align*}
    \E_{\y(w)}\Big[\exp (-
    \crochet{\cE_{\ell}, \phi(\cdot - c)}) \Big] =
    \E_{\y(w)}\Big[ 1- \sum_{u \in \cN_{\ell}} \phi( R_{\ell}^{(u)}- m_t^{(d)}-y-c) \Big] \,,
\end{align*}
where $\phi: \R \to \R$ is a continuous function and
 $w \in [\ell^{1/3}, \ell^{2/3}]$ (c.f. the definition of $G_{L,t}(v)$ in~\eqref{def:G-event}).
 We will approximate this expression via a coupling with the one-dimensional branching Brownian motion, which will then via a famous formula of McKean~\cite{McKean75} give an expression in terms of the solution to the F-KPP equation, for which asymptotics are by now well-understood (see Proposition~\ref{prop:abk-tailconvergence}). Let us discuss these matters now.

In light of the SDE for the $d$-dimensional Bessel process,
\begin{align*}
    \d R_s = \frac{\alpha_d}{R_s} \d s + \d W_s \,,
\end{align*}
one might expect that since $\y(w)$ is of order $t$, a branching Bessel process on $[0,\ell]$ started from $\y(w)$ may be coupled to be ``very close'' with probability going to $1$ as $t$ goes to infinity with a one-dimensional branching Brownian motion (for which many more results are known). Indeed, this was shown in~\cite{KLZ21}. Before stating this result, let us state the coupling.
Consider the natural coupling  of a one-dimensional BBM and a $d$-dimensional branching Bessel process obtained by using the same branching tree for both processes (hence the same set of particles in both processes at all times), and the same driving Brownian motion for each edge in the tree (to be used in each of the SDEs by the two processes for evaluating the location of the corresponding particle).
Thus, for all $s >0$, each $v \in N_s$, is associated to a Bessel process $R^{(v)}_\cdot$ and a $1$-d Brownian motion $W^{(v)}_\cdot$ satisfying the SDE
\[
dR^{(v)}_r = \frac{\alpha_d}{R^{(v)}_r} \d r + \d W^{(v)}_r \,.
\]
Note that $R^{(v)}_r \geq W^{(v)}_r$.
\begin{proposition}[{\cite{KLZ21}}]\label{prop:klz21-coupling}
Consider the above-defined  coupling of $1$-dimensional BBM $\{W_s^{(v)}\}_{s \geq 0, v \in N_s}$ and a branching $d$-dimensional Bessel process $\{R_s^{(v)}\}_{s \geq 0, v\in N_s}$ started at $x$, for some $x>0$. Fix $\ell>0$, and let
\[ \cG_x = \Big\{ \min_{v\in \cN_{\ell}} \inf_{0\leq s \leq \ell}  R_s^{(v)} \geq x/4\Big\}\,.\]
Then there exists some constant $C_d>0$ such that for large enough $x$ (in terms of $\ell$),
\[ \sup_{0 \leq  s \leq \ell} \sup_{v\in N_s} \left|R_s^{(v)} - W_s^{(v)} \right| \one_{\cG_x} \leq C_d \ell / x \,.
 \]
 Furthermore, $\P_x(\cG_x^c) \leq (2+e^{\ell})e^{-x^2/8\ell}$.
\end{proposition}
The first bound is given in~\cite[Claim~6.2]{KLZ21}, while the bound on $\P_x(\cG_x^c)$ is given in the proof of~\cite[Corollary~6.3]{KLZ21}.
 As a consequence of the coupling result in Proposition~\ref{prop:klz21-coupling}, we have the following corollary, which shows that functionals of a branching Bessel process may be replaced by functionals of one-dimensional BBM, up to a negligible error, if the time interval on which the process is considered is much shorter than the initial position of the process.
\begin{corollary}\label{cor:coupling}
Fix $\phi: \R \to \R$  a uniformly continuous function, $y \in \R$, and $c> 0$. Let $\ell:= \ell(L)$ be as in~\eqref{def:ell}. Couple $\{W_s^{(v)}\}_{s\geq 0, v\in N_s}$ and $\{R_s^{(v)}\}_{s\geq 0, v\in N_s}$ as above. Then
uniformly over $x := x(t)$ such that $x \to \infty$ as $t \to \infty$, we have
\begin{align*}
    \E_{x}\Big[\exp (-
    \crochet{\cE_{\ell}, \phi(\cdot - c)})
    - \exp\Big(\!-\!\!\!\sum_{u \in \cN_{\ell}} \phi(W_{\ell}^{(u)}- m_t^{(d)}-c) \Big)\Big]  = o_u\big(e^{-\frac{1}2 L^{3}}\big)\,.
\end{align*}
\end{corollary}

\begin{remark}
The uniform continuity assumption on $\phi$ in Corollary~\ref{cor:coupling} may be removed if the range of $x$ is taken to be any ball around  $m_t^{(d)}$ with radius given by a function of $L$.
\end{remark}

\begin{remark}
Corollary~\ref{cor:coupling} complements~\cite[Corollary~6.3]{KLZ21}, which may be thought of as Corollary~\ref{cor:coupling} with $\phi(a) := \infty \ind{a >0}$.
\end{remark}

\begin{proof}[Proof of Corollary~\ref{cor:coupling}]
Fix $L>0$, and recall that $t\to\infty$ independently of $L$.
Since $\phi$ is uniformly continuous, setting
\[
\ep_t := \sup_{a \in \R} \sup_{x \in [0,C_d \ell/t]} \abs{ \phi(a + x) - \phi(a)}\,.
\]
we have $\epsilon_t \to 0$ as $t \to \infty$.
Further, by the Markov inequality,
\[
  \P( |\cN_{\ell} | > \exp(L^3)) \leq \exp(\ell - L^3)= o(e^{-\frac{1}2 L^{3}} ).
\]
Let
\[
f(\ell,t) := \exp\Big(\!-\!\!\!\sum_{u \in \cN_{\ell}} \phi(W_{\ell}^{(u)}- m_t^{(d)}-c) \Big) \,.
\]
On the event $\cG_x \cap \{ |\cN_{\ell}| \leq \exp(L^3) \}$, we have
\begin{align*}
    \exp \big(-
    \crochet{\cE_{\ell}, \phi(\cdot - c)} \big)  \in f(\ell, t) \cdot [\exp(-\ep_t e^{L^3} ), \exp(\ep_t e^{L^3} )] \,.
\end{align*}
The above implies that
\begin{align*}
    &\bigg|\E_x\Big[ \exp \big(\!-
    \crochet{\cE_{\ell}, \phi(\cdot - c)}\big) - f(\ell,t) \Big] \bigg|\\
    &\leq
    \Big( e^{\ep_t e^{L^3}} -1 \Big)\E_x\big[  f(\ell,t) \big] + 2\Big(\P(\cG_x^c) + \P( |\cN_{\ell}| > e^{L^3})\Big)\,.
\end{align*}
Here, we have used the fact that the quantity inside the expectation on the
left hand side is bounded in absolute value by $2$. In light of $|f(\ell,t)| \leq 1$ and the bound on $\P(\cG_x^c)$ from Proposition~\ref{prop:klz21-coupling}, we find 
\[
\limsup_{t \to \infty}  \bigg|\E_x\Big[ \exp \big(\!-
    \crochet{\cE_{\ell}, \phi(\cdot - c)}\big) - f(\ell,t) \Big] \bigg|
    \leq 2\P(|\cN_{\ell}| > e^{L^3})  = o\big(e^{-\frac{1}2 L^{3}} \big)\,.
\]
This concludes the proof.
\end{proof}

\subsection{Proof of Proposition~\ref{prop:tailEstimate}}
We are now ready to prove Proposition~\ref{prop:tailEstimate}.

\begin{proof}[Proof of Proposition~\ref{prop:tailEstimate}]
Define $y$ to be the infimum of the (compact) support of $\phi$ in the second coordinate, that is,
\[
y:=
\inf \{ z \in \R :  \sup_{\theta \in \S^{d-1}} \phi(\theta,z) >0  \}\,.
\]
Define the function $\bar\phi: \S^{d-1} \times \R \to \R_{\geq 0}$ as
\[
\bar{\phi}(\sigma, x) :=\phi(\sigma, x +y)\,,
\]
so that $\bar\phi$ is supported on $\S^{d-1} \times \R_{\geq 0}$. We seek the leading-order asymptotics of
\begin{align*}
    \Upsilon := \E_{\theta (\sqrt{2}L -z )}\Big[ 1- \exp \Big( -\sum_{v \in \cN_{\tll}} \sum_{ u \in \cN_{\ell}^v} \bar\phi \big( \theta_{\tl}^{(u)}, R_{\tl}^{(u)} - m_t^{(d)} -y \big)  \Big) \Big] \,.
\end{align*}
Note that only particles $u \in \cN_{\tl}$ such that $R_{\tl}^{(u)} \geq m_t^{(d)} +y$ contribute to the exponential, since $\bar{\phi}_{\sigma}(x) := \bar\phi(\sigma,x)$ is only supported on $x \geq 0$; thus, we may add the indicator $\fT(v)$ (defined in \eqref{def:T-tailevent}) to the sum in the exponential without change.
Furthermore, Proposition~\ref{prop:sector-result} states that
\[
    \inf_{\theta \in \sd} \P_{\theta(\sqrt{2}L-z)} \bigg( \bigcap_{\{u \in \cN_{\tl}\,,~R_{\tl}^{(u)} > m_t^{(d)} + y\}} \Big\{ \|\theta - \theta_{\tl}^{(u)}\| < L^{-1/12} \Big\} \bigg) > 1- e^{-L^{5/6}}
\]
Define the function
\[
\psi_{\theta}(v) := \psi_{\theta, y, \ell, L,t}(v) = \sum_{u \in \cN_{\ell}^v} \bar\phi_{\theta}(R_{\tl}^{(u)} - m_t^{(d)}- y) \,.
\]
Then the above gives the following expression
\begin{align*}
    \Upsilon &= (1+o_u(1))\E_{\theta (\sqrt{2}L -z )}\Big[ 1- \exp \Big(\! - \!\!\!\sum_{v \in \cN_{\tll}} \psi_{\theta}(v) \indset{\fT(v)} \Big)  \Big]
    + o_u(\fM_{L,z})\,,
\end{align*}
where the $o_u(\cdot)$'s hold uniformly in $\theta$ and we used that
$\fM_{L,z}\geq e^{-cL^{2/3}}\gg e^{-L^{5/6}}$. Now,
\begin{align*}
    &
    \E_{\theta (\sqrt{2}L -z )}\Big| \exp \Big(\! - \!\!\!\sum_{v \in \cN_{\tll}} \psi_{\theta}(v) \indset{\fT(v)} \Big) - \exp \Big(\! - \!\!\!\sum_{v \in \cN_{\tll}} \psi_{\theta}(v) \indset{G_{L,t}(v) \cap \fT(v)} \Big) \Big|  \nonumber \\
    &\leq 2 \P_{\theta (\sqrt{2}L -z )} \Big( \cup_{v\in \cN_{\tll}} G_{L,t}(v)^c \cap \fT(v) \Big) \\
    &\leq 2\E_{\theta (\sqrt{2}L -z )}\Big[ \sum_{v \in \cN_{\tll}} \indset{F_{L,t}(v)^c} + \indset{F_{L,t}(v)\setminus G_{L,t}(v)} \indset{\fT(v)}\Big]
    =o_u(\fM_{L,z})\,,
\end{align*}
where the last step follows from Proposition~\ref{prop:klz21-1stmom-results}.
Thus, we have
\begin{align}
\Upsilon = (1+o_u(1)) \hat{\Upsilon} + o_u(\fM_{L,z}) \,,
\label{eqn:tailasymp-insertedbarrier}
\end{align}
where
\begin{align*}
    \hat{\Upsilon} :=
     \E_{\theta (\sqrt{2}L -z )}\Big[ 1- \exp \Big(\! - \!\!\!\sum_{v \in \cN_{\tll}} \psi_{\theta}(v) \indset{G_{L,t}(v) \cap \fT(v)} \Big)  \Big)  \Big] \,.
\end{align*}

Let $\fE_{\tll}$ denote the event that $\ind{G_{L,t}(v)\cap \fT(v)} = 1$ for at most one $v\in \cN_{\tll}$. Using the identity $1-\exp(\sum_{i=1}^n x_i) = \sum_{i=1}^n (1- e^{x_i})$ when at most one of the $x_i$ is nonzero, as well as the trivial bound $e^{-x} \leq 1$, we find
\begin{align*}
    \hat{\Upsilon} &\leq
    \E_{\theta (\sqrt{2}L -z )}\Big[ \sum_{v \in \cN_{\tll}} \big( 1- e^{-\psi_{\theta}(v) \indset{G_{L,t}(v) \cap \fT(v)}} \big) ; \fE \Big] + \P_{\theta(\sqrt{2}L-z)}(\fE_{\tll}^c) \nonumber \\
    &\leq \E_{\theta (\sqrt{2}L -z )}\Big[ \sum_{v \in \cN_{\tll}} \big( 1- e^{-\psi_{\theta}(v) \indset{G_{L,t}(v) \cap \fT(v)}} \big] + o_u(\fM_{L,z})
\end{align*}
(the last line follows from Proposition~\ref{prop:klz21-2ndmom}), as well as
\begin{align*}
    &\E_{\theta (\sqrt{2}L -z )}\Big[ \sum_{v \in \cN_{\tll}} \big( 1- e^{-\psi_{\theta}(v) \indset{G_{L,t}(v) \cap \fT(v)}} \big) \Big] \nonumber \\
    &\leq \hat{\Upsilon} +
    \E_{\theta (\sqrt{2}L -z )}\Big[ \sum_{v \in \cN_{\tll}} \big( 1- e^{-\psi_{\theta}(v) \indset{G_{L,t}(v) \cap \fT(v)}} \big) ; \fE_{\tll}^c \Big] \nonumber \\
    &\leq \hat{\Upsilon} + 2\E_{\theta (\sqrt{2}L -z )}\Big[ \sum_{v \neq w \in \cN_{\tll}} \ind{G_{L,t}(v) \cap \fT(v) \cap G_{L,t}(w) \cap \fT(w)} \Big] \nonumber
    = \hat{\Upsilon} + o_u(\fM_{L,z})
\end{align*}
(again, the last equality follows from Proposition~\ref{prop:klz21-2ndmom}).
Equation~\eqref{eqn:tailasymp-insertedbarrier} and the above inequalities then yield
\begin{align*}
    \Upsilon = (1+o_u(1))\E_{\theta (\sqrt{2}L -z )}\Big[ \sum_{v \in \cN_{\tll}} \big( 1- e^{-\psi_{\theta}(v) \indset{G_{L,t}(v) \cap \fT(v)}} \big) \Big] + o_u(\fM_{L,z})\,.
\end{align*}
Note that the quantity inside of the expectation depends only on the \textit{norms} of our BBM particles. Further, since
$\bar\phi_{\theta}$ is supported on $\R_{\geq 0}$, it follows that $\psi_{\theta}(v)$ is nonzero only on the event $\fT(v)$.
Thus, the proof of Proposition~\ref{prop:tailEstimate} will be complete if we can show the following asymptotic, uniformly over $\theta$:
\begin{align}
    \E_{\sqrt{2}L -z}\Big[ \sum_{v \in \cN_{\tll}} \big( 1- e^{-\psi_{\theta}(v) \indset{G_{L,t}(v)}} \big) \Big] \usim C_d(\phi_{\theta}) \fM_{L,z}
    \label{eqn:tailasymp-reduction}
\end{align}
With the identity $1-e^{x \indset{A}} = (1-e^x)\indset A$ and the many-to-one lemma (Lemma~\ref{lem:many-to-one}), the left-hand side of~\eqref{eqn:tailasymp-reduction} simplifies as
\begin{align*}
    &e^{\tll} \E_{\sqrt{2}L-z}\Big[ \big( 1- e^{-\psi_{\theta}(v)} \big) \indset{G_{L,t}(v)} \Big] \,.
\end{align*}
We can  expand this expectation by applying the Girsanov transform (given by~\cite[Equation~2.7]{KLZ21}) to convert the Bessel process $(R_s^{(v)})_{[0,\tll]}$
to a Brownian motion $(W_s^{(v)})_{[0,\tll]}$, and then integrating over the (Brownian) transition density $p^{W_.^{(v)}}_{\tll}$. Letting $\x(z):= \sqrt{2}L-z$, this gives the previous display as
\begin{align*}
    &e^{\tll} \E_{\sqrt{2}L-z} \left[
    \left(\frac{{W_{\tl-\ell}^{(v)}}}{\x(z)} \right)^{\alpha_d}
    \exp \left( \int_0^{\tl-\ell} \frac{\alpha_d- \alpha_d^2}{{W_s^{(v)}}^2}\d s \right) \big( 1- e^{-\psi_{\theta}(v)} \big) \indset{G_{L,t}(v)} \right]
    \\
    &\usim
    e^{\tl - \ell}
    \int_{\ell^{1/3}}^{\ell^{2/3}}
    %Gaussian density
    p_{\tl-\ell}^{W^{(v)}} \big (\x(z), \y(w) \big )
    %Girsanov term
    \bigg ( \frac{ \y(w)}{\x(z)} \bigg )^{\alpha_d}
    %barrier event
    \P_{\x(z),\tll}^{\y(w)} \left( \B_{[0,\tll]}({W_{\cdot}^{(v)}}) \right)
    %tail event
    \E_{\y(w)}[\Psi_{\theta,\ell}]\d w \,,
\end{align*}
where
\[
\Psi_{\theta,\ell} := 1- \exp (\langle \cE_{\ell}, \bar{\phi}_{\theta}(\cdot - y) \rangle) \,.
\]
In the last line, we have used the fact that on the event $G_{L,t}(v)$ (or, more specifically, on $\B_{[0,\tll]}({W_{\cdot}^{(v)}})$), we have
\[\exp \bigg( \int_0^{\tl-\ell} \frac{\alpha_d- \alpha_d^2}{{W_s^{(v)}}^2}\d s \bigg) \usim 1 \,.\]
We have also used the Markov property at time $\tll$. Now, from equations $(6.53), (6.54)$ and $(6.60)$ of~\cite{KLZ21}, the last display is asymptotically equivalent (in the sense of $\usim$) to
\begin{align*}
    &\sqrt{\frac{2^{1+\alpha_d}}{\pi}} e^{-y\sqrt{2}}\fM_{L,z}
    \int_{\ell^{1/3}}^{\ell^{2/3}} w e^{w \sqrt{2}} \E_{\y(w)}[ \Psi_{\theta, \ell}] \d w \,.
\end{align*}
Corollary~\ref{cor:coupling} and Proposition~\ref{prop:abk-tailconvergence} tell us that
\begin{align*}
    \int_{\ell^{1/3}}^{\ell^{2/3}} w e^{w \sqrt{2}} \E_{\y(w)}[ \Psi_{\theta, \ell}] \d w \usim e^{\sqrt{2}y} C(\phi_{\theta})\,,
\end{align*}
where $C(\phi_{\theta})$ is defined in~\eqref{def:C(phi)}. This concludes the proof.
\end{proof}

\section*{Acknowledgments}
We thank two anonymous referees for their comments, which led to a significant
improvement. In particular, Proposition 1.3 was added in response to their comments.
Y.H.K.\ was supported by the NSF Graduate Research Fellowship 1839302.
E.L.\ was supported by NSF grants DMS-1812095 and DMS-2054833 and by US-BSF grant 2018088.
B.M. is partially supported by the ANR grant MALIN (ANR-16-CE93-0003).
O.Z.\ was partially supported by  US-BSF grant 2018088
 and by the European Research Council (ERC) under the European Union's Horizon 2020 research and innovation programme (grant agreement No.~692452).
This material is based upon work supported by the NSF under Grant No. DMS-1928930 while Y.H.K.\ and O.Z.\ participated in a program hosted by MSRI in Berkeley, California, during the Fall 2021 semester.

\bibliographystyle{imsart-number}
\bibliography{bbmpp}

\end{document}